\documentclass[a4paper]{amsart}

\usepackage[backrefs,lite]{amsrefs}

\usepackage{amssymb}              
\usepackage{lmodern,bm}              
\usepackage{pifont}
\usepackage{mathtools}
\usepackage[all]{xy}
\usepackage{tikz}
\usetikzlibrary{calc} 

\usepackage{longtable}
\usepackage{array}
\usepackage{multicol}
\usepackage{multirow}
\usepackage{makecell}
\usepackage{booktabs}
\newcolumntype{C}{>{$}c<{$}}
\newcolumntype{L}{>{$}l<{$}}
\newcolumntype{R}{>{$}r<{$}}

\CompileMatrices

\newtheorem{theorem}{Theorem}[section]

\newtheorem{lemma}[theorem]{Lemma}
\newtheorem{proposition}[theorem]{Proposition}
\newtheorem{corollary}[theorem]{Corollary}

\theoremstyle{definition}

\newtheorem{definition}[theorem]{Definition}
\newtheorem{construction}[theorem]{Construction}

\newtheorem{example}[theorem]{Example}
\newtheorem{remark}[theorem]{Remark}
\newtheorem{reminder}[theorem]{Reminder}

\numberwithin{equation}{theorem}

\def\KK{{\mathbb K}}
\def\TT{{\mathbb T}}
\def\ZZ{{\mathbb Z}}

\def\QQ{{\mathbb Q}}
\def\PP{{\mathbb P}}

\def\KKK{\mathcal{K}}

\def\bangle#1{{\langle #1 \rangle}}

\def\quot{/\!\!/}

\def\reg{\mathrm{reg}}

\DeclareMathOperator{\Cl}{\mathrm{Cl}}
\DeclareMathOperator{\cone}{\mathrm{cone}}
\DeclareMathOperator{\conv}{\mathrm{conv}}

\DeclareMathOperator{\ddiv}{\mathrm{div}}
\DeclareMathOperator{\im}{\mathrm{im}}
\DeclareMathOperator{\lcm}{\mathrm{lcm}}

\DeclareMathOperator{\LP}{\mathrm{LP}}
\DeclareMathOperator{\rank}{\mathrm{rank}}

\DeclareMathOperator{\Spec}{\mathrm{Spec}}
\DeclareMathOperator{\supp}{\mathrm{supp}}
\DeclareMathOperator{\trop}{\mathrm{trop}}
\DeclareMathOperator{\WDiv}{\mathrm{WDiv}}

\title[Anticanonical complex for toric complete intersections]{The anticanonical complex for non-degenerate toric complete intersections}

\author[J\"urgen Hausen, Christian Mauz, Milena Wrobel]{J\"urgen Hausen, Christian Mauz, Milena Wrobel}

\address{Mathematisches Institut, Universit\"at T\"ubingen,
Auf der Morgenstelle 10, 72076 T\"ubingen, Germany}
\email{juergen.hausen@uni-tuebingen.de}

\address{Mathematisches Institut, Universit\"at T\"ubingen,
Auf der Morgenstelle 10, 72076 T\"ubingen, Germany}
\email{mauz@math.uni-tuebingen.de}

\address{Institut f\"ur Mathematik, Universit\"at Oldenburg,
26111 Oldenburg, Germany}
\email{milena.wrobel@uni-oldenburg.de}


\begin{document}

\begin{abstract}
The anticanonical complex generalizes
the Fano polytope from toric geometry
and has been used to study Fano varieties
with torus action so far.
We work out the case of complete intersections 
in toric varieties defined by non-degenerate
systems of Laurent polynomials.
As an application, we classify the terminal
Fano threefolds that are embedded into a fake
weighted projective space via a general 
system of Laurent polynomials.
\end{abstract}

\maketitle

\section{Introduction}

The idea behind anticanonical complexes is to
extend the features of the Fano polytopes
from toric geometry to wider classes of
varieties and thereby to provide combinatorial
tools for the treatment of the singularities
of the minimal model programme.
If $X$ is any $\QQ$-Gorenstein variety,
i.e.~some positive multiple of a canonical
divisor $K_X$ is Cartier,
then these singularities are defined 
in terms of \emph{discrepancies} that 
means the coefficients $a(E)$ of the 
exceptional divisors $E$ showing up
in the ramification formula for a
resolution $\pi \colon X' \to X$ of 
singularities: 
$$
K_{X'}
\ = \ 
\pi^* K_{X} + \sum a(E)E.
$$
The variety~$X$ has at most
\emph{terminal, canonical} or 
\emph{log terminal}
singularities if always $a(E) > 0$,
$a(E) \ge 0$ or $a(E) > -1$.
We briefly look at the toric
case.
For an $n$-dimensional toric Fano variety~$Z$,
one defines the \emph{Fano polytope} to be 
the convex hull $A \subseteq \QQ^n$
over the primitive ray generators of 
the describing fan of~$Z$.
For any toric resolution 
$\pi \colon Z' \to Z$ 
of singularities, the exceptional 
divisors~$E_\varrho$ are given 
by rays of the fan of~$Z'$
and one obtains the discrepancies 
as
$$ 
a(E_\varrho) 
\ = \ 
\frac{\Vert v_\varrho \Vert}{\Vert v_\varrho'  \Vert} - 1,
$$
where $v_\varrho \in \varrho$ is the shortest
non-zero lattice vector and $v_\varrho' \in \varrho$ 
is the intersection point of $\varrho$ and 
the boundary $\partial A$.  
In particular, a toric Fano variety~$Z$
is always log terminal and $Z$ has at most
terminal (canonical) singularities
if and only if its corresponding Fano polytope $A$
contains no lattice points except the origin
and its vertices (no lattice points in its
interior except the origin).
This allows the use of lattice polytope
methods in the study of singular toric
Fano varieties; see~\cite{BoBo,Ka1,Ka2}
for work in this direction.

This principle has been extended by 
replacing the Fano polytope with a 
suitable polyhedral complex, named
\emph{anticanonical complex} in the 
setting of varieties with a torus action 
of complexity one, which encodes 
discrepancies in full analogy to the 
toric Fano polytope; see~\cite{BeHaHuNi}.
The more recent work~\cite{HiWr1}
provides an existence result of 
anticanonical complexes for torus actions 
of higher complexity
subject to conditions on a rational quotient.
Applications to the study of singularities 
and Fano varieties can be found 
in~\cite{ArBrHaWr,BrHae,HiWr2}. 

In the present article, we provide 
an anticanonical complex 
for subvarieties of toric varieties 
arising from non-degenerate systems of 
Laurent polynomials in the sense of
Khovanskii~\cite{Kh}; see also 
Definition~\ref{def:laurent-system}.
Even in the hypersurface case, the
subvarieties obtained
this way form an interesting example
class of varieties which is actively
studied by several authors;
see for instance~\cite{Ba1,Dol,Ish}.

We briefly indicate the setting;
see Section~\ref{sec:laurent-systems} 
for the details.
Let $F = (f_1, \ldots, f_s)$ 
be a non-degenerate system of Laurent 
polynomials in $n$ variables and let 
$\Sigma$ be any fan in $\ZZ^n$ refining
the normal fan of the Minkowski sum
$B_1 + \ldots + B_s$ of the Newton 
polytopes $B_j$ of $f_j$.
Moreover, denote by $Z$ the toric variety 
associated with $\Sigma$.
We are interested in the 
\emph{non-degenerate toric complete intersection}
defined by $F$ and~$\Sigma$, that means
the variety
$$
X \ = \ X_1 \cap \ldots \cap X_s \ \subseteq Z,
$$
where $X_i \subseteq Z$ is the closure of
$V(f_i) \subseteq \TT^n$.
By Theorem~\ref{thm:nonde2ci}, the variety
$X \subseteq Z$ is a locally complete intersection,
equals the closure of $V(F) \subseteq \TT^n$
and, in the Cox ring of $Z$, the defining
homogeneous equations of $X$ generate a
complete intersection ideal.
Theorem~\ref{thm:nondeg2trop} shows that the
union $Z_X \subseteq Z$ of all torus orbits
intersecting $X$ is open in $Z$ and thus the
corresponding cones form a subfan
$\Sigma_X \subseteq \Sigma$.
Moreover, the support of $\Sigma_X$ equals
the tropical variety of $V(F) \subseteq \TT^n$.

We come to the first main result of the article.
Suppose that $Z_X$ is $\QQ$-Gorenstein.
Then, for every $\sigma \in \Sigma_X$, we have 
a linear form $u_\sigma \in \QQ^n$ evaluating 
to $-1$ on every primitive ray generator
$v_\varrho$, where $\varrho$ is an extremal
ray of $\sigma$.
We set   
$$ 
A(\sigma) 
\ := \ 
\{v \in \sigma; \ 0 \ge \bangle{u_\sigma,v} \ge -1\}
\ \subseteq \
\sigma.
$$

\goodbreak

\begin{theorem}
\label{thm:main-1}
Let $X \subseteq Z$ be an irreducible
non-degenerate toric complete intersection.
Then $X \subseteq Z$
admits ambient toric resolutions.
Moreover, if $Z_X$ is $\QQ$-Gorenstein,
then~$X$ is so and $X$ has an anticanonical
complex
$$
\mathcal{A}_X
\ = \
\bigcup_{\sigma \in \Sigma_X} A(\sigma).
$$
That means that for all ambient toric 
modifications $Z' \to Z$ the discrepancy
of any exceptional divisor $E_{X'} \subseteq X'$ 
is given in terms of the defining ray
$\varrho \in \Sigma'$ of its host 
$E_{Z'} \subseteq Z'$,
the primitive generator $v_\varrho \in \varrho$
and the intersection point $v_\varrho'$
of $\varrho$ and the boundary
$\partial \mathcal{A}_X$ as
$$ 
a (E_{X'}) 
\ = \ 
\frac{\Vert v_\varrho \Vert}{\Vert v_\varrho' \Vert} -1.
$$
\end{theorem}

Observe that in the above setting, each vertex 
of $\mathcal{A}_X$ is a 
primitive ray generator of the fan $\Sigma$. 
Thus, in the non-degenerate complete toric intersection  
case, all vertices of the anticanonical complex 
are integral vectors;
this does definitely not hold in other situations, 
see~\cite{BeHaHuNi,HiWr1}.
The following consequence of Theorem~\ref{thm:main-1}
yields in particular Bertini type statements on
terminal and canonical singularities.

\begin{corollary}
\label{cor:acc-singu}
Consider a subvariety $X \subseteq Z$ as in Theorem~\ref{thm:main-1}
and the associated anticanonical complex~$\mathcal{A}_X$.
\begin{enumerate}
\item
$X$ has at most log-terminal singularities.
\item
$X$ has at most terminal singularities 
if and only if $\mathcal{A}_X$ contains no lattice
points except the origin and its vertices.
\item
$X$ has at most canonical singularities 
if and only if $\mathcal{A}_X$ contains no 
interior lattice points except the origin.
\end{enumerate}
Moreover, $X$ has at most terminal (canonical)
singularities if and only if its ambient 
toric variety $Z_X$ has at most terminal (canonical)
singularities.
\end{corollary}

As an application of the first main result,
we classify the general non-toric terminal Fano
non-degenerate complete intersection
threefolds sitting in fake weighted
projective spaces; for the meaning of
``general'' in this context, see
Definition~\ref{def:general}.
According to~\cite{Kh}, the general toric
complete intersection is non-degenerate.
Moreover, under suitable assumptions
on the ambient toric variety, we
obtain the divisor class group and
the Cox ring for free in the general
case; see Corollary~\ref{cor:general2SrRa}.
This, by the way, allows us to construct
many Mori dream spaces with prescribed
properties; see for instance
Example~\ref{ex:gentype}. 

We turn to the second main result.
Recall that a fake weighted projective space
is an $n$-dimensional toric variety
arising from a complete fan with $n+1$ rays.
Any fake weighted projective space~$Z$ is
uniquely determined up to isomorphism
by its degree matrix $Q$, having as its
columns the divisor classes $[D_i] \in \Cl(Z)$
of the toric prime divisors $D_1, \ldots, D_{n+1}$
of~$Z$.

\goodbreak

\begin{theorem}
\label{thm:terminalWHf-new}
Any non-toric terminal Fano general complete
intersection threefold $X = X_1 \cap \ldots \cap X_s$
in a fake weighted projective space $Z$ is a member of
precisely one of the following families,
specified by the generator degree matrix~$Q$
and the relation degree matrix $\mu$ with respect to
the $\Cl(Z)$-grading.
We also list the anticanonical class
$-\mathcal{K}$ of $X$ and the numbers
$-\mathcal{K}^3$ and $h^0(-\mathcal{K})$.

\begin{center}
\newcounter{terminalWHfNo}
\begin{longtable}{>{\refstepcounter{terminalWHfNo}\theterminalWHfNo }LCCCCCC}
\toprule
\multicolumn{1}{c}{No.}
&
\Cl(Z)
& Q
& \mu
& -\mathcal{K}
& -\mathcal{K}^3
& h^0(-\KKK)
\\
\midrule
& \multirow{3}{*}{$\ZZ$} & \multirow{3}{*}{$
	{\small \setlength{\arraycolsep}{3pt}\begin{bmatrix}
	1 & 1 & 1 & 1 & 1
	\end{bmatrix}} $}
    & 2 & 3 & 54 & 30 \\
& & & 3 & 2 & 24 & 15 \\
& & & 4 & 1 &  4 &  5 \\ \midrule

& \ZZ \times \ZZ / 3 \ZZ &
	{\small \setlength{\arraycolsep}{3pt}\begin{bmatrix}
	1 & 1 & 1 & 1 & 1 \\
	\bar{0} & \bar{0} & \bar{1} & \bar{1} & \bar{2} 
	\end{bmatrix}} 
	& {\small \setlength{\arraycolsep}{3pt}
          \begin{bmatrix}
            3 \\
            \bar{0} \\
          \end{bmatrix}} 
	& {\small \setlength{\arraycolsep}{3pt}
          \begin{pmatrix}
            2 \\
            \bar{1} \\
          \end{pmatrix}} &
	 8 & 5 \\ \midrule

& \ZZ & {\small \setlength{\arraycolsep}{3pt}\begin{bmatrix}
	1 & 1 & 1 & 1 & 2
	\end{bmatrix}} 
	& 4 & 2 & 16 & 11 \\ \midrule

& \ZZ \times \ZZ / 2 \ZZ &
	{\small \setlength{\arraycolsep}{3pt}\begin{bmatrix}
	1 & 1 & 1 & 1 & 2 \\
	\bar{0} & \bar{0} & \bar{1} & \bar{1} & \bar{1}
	\end{bmatrix}} 
	& {\small \setlength{\arraycolsep}{3pt}
          \begin{bmatrix}
            4 \\
            \bar{0} \\
          \end{bmatrix}}
	& {\small \setlength{\arraycolsep}{3pt}
          \begin{pmatrix}
            2 \\
            \bar{1} \\
          \end{pmatrix}} 
	 & 8 & 5 \\ \midrule

& \multirow{2}{*}{$\ZZ$} & \multirow{2}{*}{$
	{\small \setlength{\arraycolsep}{3pt}\begin{bmatrix}
	1 & 1 & 1 & 2 & 2
	\end{bmatrix}} $}
    & 4 & 3 & 27 & 16 \\
& & & 6 & 1 & 3/2 & 3 \\ \midrule

& \ZZ \times \ZZ / 2 \ZZ &
	{\small \setlength{\arraycolsep}{3pt}\begin{bmatrix}
	1 & 1 & 1 & 2 & 2 \\
	\bar{0} & \bar{0} & \bar{1} & \bar{1} & \bar{1}
	\end{bmatrix}} 
	& {\small \setlength{\arraycolsep}{3pt}
          \begin{bmatrix}
            4 \\
            \bar{0} \\
          \end{bmatrix}}
	& {\small \setlength{\arraycolsep}{3pt}
          \begin{pmatrix}
            3 \\
            \bar{1} \\
          \end{pmatrix}} 
        &  27/2 & 8 \\ \midrule

& \ZZ \times \ZZ / 3 \ZZ &
	{\small \setlength{\arraycolsep}{3pt}\begin{bmatrix}
	1 & 1 & 1 & 2 & 2 \\
	\bar{0} & \bar{1} & \bar{2} & \bar{0} & \bar{1}
	\end{bmatrix}} 
      &
      {\small \setlength{\arraycolsep}{3pt}
          \begin{bmatrix}
            6 \\
            \bar{0} \\
          \end{bmatrix}} 
	& {\small \setlength{\arraycolsep}{3pt}
          \begin{pmatrix}
            1 \\
            \bar{1} \\
          \end{pmatrix}} 
        & 1/2 & 1 \\ \midrule

& \ZZ & {\small \setlength{\arraycolsep}{3pt}\begin{bmatrix}
	1 & 1 & 1 & 1 & 3
	\end{bmatrix}} 
	& 6 & 1 & 2 & 4 \\ \midrule

& \ZZ & {\small \setlength{\arraycolsep}{3pt}\begin{bmatrix}
	1 & 1 & 1 & 2 & 3
	\end{bmatrix}} 
	& 6 & 2 & 8 & 7 \\ \midrule

& \ZZ & {\small \setlength{\arraycolsep}{3pt}\begin{bmatrix}
	1 & 1 & 2 & 2 & 3
	\end{bmatrix}} 
	& 6 & 3 & 27/2 & 9 \\ \midrule

& \ZZ & {\small \setlength{\arraycolsep}{3pt}\begin{bmatrix}
	1 & 1 & 2 & 3 & 3
	\end{bmatrix}} 
	& 6 & 4 & 64/3 & 13 \\ \midrule

& \ZZ & {\small \setlength{\arraycolsep}{3pt}\begin{bmatrix}
	1 & 2 & 2 & 3 & 3
	\end{bmatrix}} 
	& 6 & 5 & 125/6 & 12 \\ \midrule

& \ZZ & {\small \setlength{\arraycolsep}{3pt}\begin{bmatrix}
	1 & 1 & 1 & 2 & 4
	\end{bmatrix}} 
	& 8 & 1 & 1 & 3 \\ \midrule

        & \ZZ \times \ZZ / 2 \ZZ
        &
        {\small \setlength{\arraycolsep}{3pt}
          \begin{bmatrix}
            1 & 1 & 1 & 2 & 4 \\
            \bar{0} & \bar{0} & \bar{1} & \bar{1} & \bar{1} \\
          \end{bmatrix}} 
        &
        {\small \setlength{\arraycolsep}{3pt}
          \begin{bmatrix}
            8 \\
            \bar{0} \\
          \end{bmatrix}}
	& {\small \setlength{\arraycolsep}{3pt}
          \begin{pmatrix}
            1 \\
            \bar{1} \\
          \end{pmatrix}} 
      & 1/2 & 1 \\ \midrule

& \ZZ & {\small \setlength{\arraycolsep}{3pt}\begin{bmatrix}
	1 & 2 & 3 & 3 & 4
	\end{bmatrix}} 
	& 12 & 1 &  1/6 & 1 \\ \midrule

& \ZZ & {\small \setlength{\arraycolsep}{3pt}\begin{bmatrix}
	1 & 1 & 3 & 4 & 4
	\end{bmatrix}} 
	& 12 & 1 & 1/4 & 2 \\ \midrule

& \ZZ & {\small \setlength{\arraycolsep}{3pt}\begin{bmatrix}
	1 & 1 & 2 & 2 & 5
	\end{bmatrix}} 
	& 10 & 1 & 1/2 & 2 \\ \midrule

& \ZZ & {\small \setlength{\arraycolsep}{3pt}\begin{bmatrix}
	1 & 1 & 2 & 3 & 6
	\end{bmatrix}} 
	& 12 & 1 & 1/3 & 2 \\ \midrule

        & \ZZ \times \ZZ / 2 \ZZ
        &
        {\small \setlength{\arraycolsep}{3pt}
          \begin{bmatrix}1 & 1 & 2 & 3 & 6 \\
            \bar{0} & \bar{1} & \bar{1} & \bar{0} & \bar{1} \\
          \end{bmatrix}} 
      &{\small \setlength{\arraycolsep}{3pt}\begin{bmatrix}
          12 \\ \bar{0} \end{bmatrix}}
	& {\small \setlength{\arraycolsep}{3pt}
          \begin{pmatrix}
            1 \\
            \bar{1} \\
          \end{pmatrix}} 
 & 1/6 & 1 \\ \midrule

& \ZZ & {\small \setlength{\arraycolsep}{3pt}\begin{bmatrix}
	1 & 1 & 1 & 4 & 6
	\end{bmatrix}} 
	& 12 & 1 & 1/2 &  3 \\ \midrule

& \ZZ & {\small \setlength{\arraycolsep}{3pt}\begin{bmatrix}
	1 & 1 & 2 & 6 & 9
	\end{bmatrix}} 
	& 18 & 1 & 1/6 & 2 \\ \midrule

& \ZZ & {\small \setlength{\arraycolsep}{3pt}\begin{bmatrix}
	1 & 1 & 4 & 5 & 10
	\end{bmatrix}} 
	& 20 & 1 & 1/10 & 2 \\ \midrule

& \ZZ & {\small \setlength{\arraycolsep}{3pt}\begin{bmatrix}
	1 & 1 & 3 & 8 & 12
	\end{bmatrix}} 
	& 24 & 1 & 1/12 & 2 \\ \midrule

& \ZZ & {\small \setlength{\arraycolsep}{3pt}\begin{bmatrix}
	1 & 2 & 3 & 10 & 15
	\end{bmatrix}} 
	& 30 & 1 & 1/30 & 1 \\ \midrule

& \ZZ & {\small \setlength{\arraycolsep}{3pt}\begin{bmatrix}
	1 & 1 & 6 & 14 & 21
	\end{bmatrix}} 
	& 42 & 1 & 1/42 & 2 \\ \midrule

& \multirow{2}{*}{$\ZZ$} & \multirow{2}{*}{$
	{\small \setlength{\arraycolsep}{3pt}\begin{bmatrix}
	1 & 1 & 1 & 1 & 1 & 1
	\end{bmatrix}} $}
    & {\small \setlength{\arraycolsep}{3pt}\begin{bmatrix}
        2 & 2 \end{bmatrix}} 
        & 2 &  32 & 19 \\
        & & &  {\small \setlength{\arraycolsep}{3pt}\begin{bmatrix}
            2 & 3 \end{bmatrix}}
        & 1 & 6 & 6 \\ \midrule

& \ZZ \times \ZZ / 2 \ZZ & {\small \setlength{\arraycolsep}{3pt}\begin{bmatrix}
	1 & 1 & 1 & 1 & 1 & 1 \\
	\bar{0} & \bar{0} &\bar{0} & \bar{1} &  \bar{1} &  \bar{1}
	\end{bmatrix}} 
      &  {\small \setlength{\arraycolsep}{3pt}\begin{bmatrix}
          2 & 2 \\
          \bar{0} & \bar{0} \\         
        \end{bmatrix}}
	& {\small \setlength{\arraycolsep}{3pt}
          \begin{pmatrix}
            2 \\
            \bar{1} \\
          \end{pmatrix}} 
      & 16 & 9 \\ \midrule

& \ZZ & {\small \setlength{\arraycolsep}{3pt}\begin{bmatrix}
	1 & 1 & 1 & 2 & 2 & 2 \\
	\end{bmatrix}} 
	& {\small \setlength{\arraycolsep}{3pt}\begin{bmatrix}
          4 & 4 \\       
        \end{bmatrix}}
        & 1 &  2 & 3 \\ \midrule

& \ZZ \times \ZZ / 2 \ZZ & {\small \setlength{\arraycolsep}{3pt}\begin{bmatrix}
	1 & 1 & 1 & 2 & 2 & 2 \\
	\bar{0} & \bar{0} & \bar{1} & \bar{0} & \bar{1} & \bar{1}
	\end{bmatrix}} 
	& {\small \setlength{\arraycolsep}{3pt}\begin{bmatrix}
          4 & 4 \\
          \bar{0} & \bar{0} \\         
        \end{bmatrix}}
	& {\small \setlength{\arraycolsep}{3pt}
          \begin{pmatrix}
            1 \\
            \bar{1} \\
          \end{pmatrix}} 
            & 1 & 1 \\ \midrule

& \ZZ \times \ZZ / 2 \ZZ & {\small \setlength{\arraycolsep}{3pt}\begin{bmatrix}
	1 & 1 & 1 & 2 & 2 & 2 \\
	\bar{0} & \bar{0} & \bar{1} & \bar{1} & \bar{1} & \bar{1}
	\end{bmatrix}} 
	& {\small \setlength{\arraycolsep}{3pt}\begin{bmatrix}
          4 & 4 \\
          \bar{0} & \bar{0} \\         
        \end{bmatrix}}
	& {\small \setlength{\arraycolsep}{3pt}
          \begin{pmatrix}
            1 \\
            \bar{0} \\
          \end{pmatrix}} 
        & 1 & 2 \\ \midrule

& \ZZ \times (\ZZ / 2 \ZZ)^2 & {\small \setlength{\arraycolsep}{3pt}\begin{bmatrix}
	1 & 1 & 1 & 2 & 2 & 2 \\
	\bar{0} & \bar{0} & \bar{1} & \bar{0} & \bar{1} & \bar{1} \\
	\bar{0} & \bar{1} & \bar{0} & \bar{1} & \bar{0} & \bar{1}
	\end{bmatrix}} 
      & {\small \setlength{\arraycolsep}{3pt}\begin{bmatrix}
          4 & 4 \\
          \bar{0} & \bar{0} \\
          \bar{0} & \bar{0} \\
        \end{bmatrix}} 
      % \makecell{(4, \bar{0}, \bar{0}),\\ (4, \bar{0}, \bar{0})\phantom{,}}
	& {\small \setlength{\arraycolsep}{3pt}
          \begin{pmatrix}
            1 \\
            \bar{1} \\ 
	    \bar{1} \\
          \end{pmatrix}} 
      & 1/2 & 0 \\ \midrule

& \ZZ & {\small \setlength{\arraycolsep}{3pt}\begin{bmatrix}
	1 & 2 & 2 & 2 & 3 & 3 \\
	\end{bmatrix}} 
	& {\small \setlength{\arraycolsep}{3pt}
          \begin{bmatrix}
            6 & 6 \\
          \end{bmatrix}}
        & 1 & 1/2 & 1 \\ \midrule

& \ZZ & {\small \setlength{\arraycolsep}{3pt}\begin{bmatrix}
	1 & 1 & 2 & 3 & 3 & 3 \\
	\end{bmatrix}} 
	& {\small \setlength{\arraycolsep}{3pt}
          \begin{bmatrix}
            6 & 6 \\
          \end{bmatrix}}
        & 1 & 2/3 & 2 \\ \midrule

& \ZZ & {\small \setlength{\arraycolsep}{3pt}\begin{bmatrix}
	1 & 2 & 2 & 3 & 3 & 3 \\
	\end{bmatrix}} 
	& {\small \setlength{\arraycolsep}{3pt}
          \begin{bmatrix}
            6 & 6 \\
            \end{bmatrix}}
            & 2 & 8/3 & 3 \\ \midrule

        & \ZZ & {\small \setlength{\arraycolsep}{3pt}
          \begin{bmatrix}
	1 & 1 & 1 & 1 & 1 & 1 & 1 \\
	\end{bmatrix}} 
      & {\small \setlength{\arraycolsep}{3pt}
        \begin{bmatrix}
         2 & 2 & 2 \\
       \end{bmatrix}}
        & 1 & 8 & 7 \\ \midrule

        & \ZZ \times \ZZ / 2 \ZZ & {\small \setlength{\arraycolsep}{3pt}
          \begin{bmatrix} 
	1 & 1 & 1 & 1 & 1 & 1 & 1 \\
	\bar{0} & \bar{0} & \bar{0} & \bar{0} & \bar{1} & \bar{1} & \bar{1}
	\end{bmatrix}} 
      &{\small \setlength{\arraycolsep}{3pt}
        \begin{bmatrix}
         2 & 2 & 2 \\
         \bar{0} & \bar{0} & \bar{0} \\
        \end{bmatrix}} 
	& {\small \setlength{\arraycolsep}{3pt}
          \begin{pmatrix}
            1 \\
            \bar{1} \\
          \end{pmatrix}} 
      & 4 & 3 \\ \midrule

      & \ZZ \times (\ZZ / 2 \ZZ)^2 & {\small \setlength{\arraycolsep}{3pt}
        \begin{bmatrix}
	1 & 1 & 1 & 1 & 1 & 1 & 1 \\
	\bar{0} & \bar{0} & \bar{0} & \bar{0} & \bar{1} & \bar{1} & \bar{1} \\
	\bar{0} & \bar{0} & \bar{1} & \bar{1} & \bar{0} & \bar{0} & \bar{1} \\
	\end{bmatrix}} 
      &{\small \setlength{\arraycolsep}{3pt} \begin{bmatrix}
         2 & 2 & 2 \\
         \bar{0} & \bar{0} & \bar{0} \\
         \bar{0} & \bar{0} & \bar{0} \\
        \end{bmatrix}} 
	& {\small \setlength{\arraycolsep}{3pt}
          \begin{pmatrix}
            1 \\
            \bar{1} \\
            \bar{1} \\
          \end{pmatrix}} 
      & 2 & 1
\\
\midrule
&
\ZZ \times (\ZZ / 2 \ZZ)^3
&
{\small \setlength{\arraycolsep}{3pt}
  \begin{bmatrix}
    1 & 1 & 1 & 1 & 1 & 1 & 1
    \\
    \bar{0} & \bar{0} & \bar{0} & \bar{0} & \bar{1} & \bar{1} & \bar{1}
    \\
    \bar{0} & \bar{0} & \bar{1} & \bar{1} & \bar{0} & \bar{0} & \bar{1}
    \\
    \bar{0} & \bar{1} & \bar{0} & \bar{1} & \bar{0} & \bar{1} & \bar{0}
    \\
  \end{bmatrix}} 
&
{\small \setlength{\arraycolsep}{3pt}
  \begin{bmatrix}
    2 & 2 & 2
    \\
    \bar{0} & \bar{0} & \bar{0}
    \\
    \bar{0} & \bar{0} & \bar{0}
    \\
    \bar{0} & \bar{0} & \bar{0}
    \\
  \end{bmatrix}}
	& {\small \setlength{\arraycolsep}{3pt}
          \begin{pmatrix}
            1 \\
            \bar{1} \\
            \bar{1} \\
            \bar{1} \\
          \end{pmatrix}} 
&
1
&
0
\\
\bottomrule
\end{longtable}
\end{center}
\vskip -1cm
Moreover, each of these constellations defines non-degenerate
toric complete intersections with at most terminal singularities
in a fake weighted projective space.
In addition, for the divisor class groups,
we have $\Cl(X) = \Cl(Z)$ and 
the Cox ring of~$X$ is given by
\begin{eqnarray*}
\mathcal{R}(X)
& = &
\KK[T_1, \ldots, T_{s+4}] / \bangle{g_1, \ldots, g_s},
\\
\deg(T_i)
& = &
[D_i] \in \Cl(Z),
\\
\deg(g_j)
& = &
[X_j] \in \Cl(Z),
\end{eqnarray*}
where
$\KK[T_1, \ldots, T_{s+4}] = \mathcal{R}(Z)$
is the Cox ring of the fake weighted projective
space $Z$ and $g_1, \ldots, g_s \in \mathcal{R}(Z)$
are the defining $\Cl(Z)$-homogeneous polynomials
for $X_1, \ldots, X_s \subseteq Z$.
Moreover, $T_1, \dotsc, T_{s+4}$ define a minimal system of
prime generators for $\mathcal{R}(X)$.
\end{theorem}

\goodbreak

We note some observations around this classification
and link to the existing literature.

\begin{remark}
The toric terminal Fano complete intersection threefolds
in a fake weighted projective space 
are precisely the three-dimensional terminal fake
weighted projective spaces;
up to isomorphy, there are eight of them~\cite{Ka3}.
\end{remark}

Recall that the \emph{Fano index} of $X$ is the
maximal positive integer $q_X$ such that
$\KKK = q_XD$ with a Weil divisor~$D$ on~$X$.

\begin{remark}
For the $X$ of Theorem~\ref{thm:terminalWHf-new}
with $\Cl(Z)$ torsion free,
we have $q_X = -\mathcal{K}$,
regarding $-\mathcal{K} \in  \Cl(Z) = \ZZ$
as an integer.
In the remaining cases, $q_X$ is given by
\begin{center}
\begin{tabular}{cccccccccccccc}
No. & 4 & 6 & 9 & 10 & 17 & 22 & 31 & 33 & 34 & 35 & 40 & 41 & 42
\\
\hline
$q_X$ & 2 & 1 & 3 & 1 & 1 & 1 & 1 & 1 & 1 & 1 & 1 & 1 & 1  
\end{tabular}
\end{center}
\end{remark}

\begin{remark}
Embeddings into weighted projective spaces have
been intensely studied by several authors.
Here is how Theorem~\ref{thm:terminalWHf-new}
relates to well-knwon classifications in this
case.
\begin{enumerate}
\item
Numbers~1, 2, 3, 5, 11, 12, 29, 30 and~39 from
Theorem~\ref{thm:terminalWHf-new}
are smooth and thus appear in the classification
of smooth Fano threefolds
of Picard number one~\cite[\S~12.2]{IsPr}.
\item
Every variety $X$ from 
Theorem~\ref{thm:terminalWHf-new} with Fano index
$q_X = 1$ defined by at most two equations in
a weighted projective space~$Z$ occurs
in~\cite[Lists~16.6, 16.7]{IF}.
\item
The items from~\cite[Lists~16.6, 16.7]{IF} which
don't show up in Theorem~\ref{thm:terminalWHf-new}
are note realizable as general 
complete intersections in a fake weighted
projective space.
\end{enumerate}
\end{remark}

Recall that the \emph{Gorenstein index} of a
$\QQ$-Gorenstein variety $X$ is the minimal
positive integer $\imath_X$ such that $\imath_X K_X$
is a Cartier divisor. So, $\imath_X = 1$ means
that $X$ is Gorenstein.

\begin{remark}
The Gorenstein varieties in Theorem~\ref{thm:terminalWHf-new}
are precisely the smooth ones.
This is a direct application of Corollary~\ref{cor:SigmaXcrit}
showing that $Z_X$ is the union of all torus orbits
of dimension at least three and Proposition~\ref{prop:adjunction}
which ensures that~$X$ and $Z_X$ have the same Gorenstein index.
\end{remark}

\begin{remark}
The anticanonical self intersction $-\mathcal{K}^3$
together with the first coefficients of the Hilbert
series of $X$ from Theorem~\ref{thm:terminalWHf-new}
with $\Cl(Z)$ having torsion occur in the Graded
Ring Database~\cite{GRDB,ABR}. Here are the corresponding
ID's:
\begin{center}
\begin{tabular}{cccccccccccccc}
No.
& 4
& 6
& 9
& 10
& 17
& 22
& 31
& 33
& 34
& 35
& 40
& 41
& 42
\\
\hline
ID
& $\vcenter{\hbox{\tiny 40245}}$
& $\vcenter{\hbox{\tiny 23386}}$
& $\vcenter{\hbox{\tiny 41176}}$ 
& $\vcenter{\hbox{\tiny 2122}}$ 
& $\vcenter{\hbox{\tiny 3508}}$ 
& $\vcenter{\hbox{\tiny 1249}}$ 
& $\vcenter{\hbox{\tiny 32755}}$ 
& $\vcenter{\hbox{\tiny 4231}}$ 
& $\vcenter{\hbox{\tiny 5720}}$ 
& $\vcenter{\hbox{\tiny 237}}$ 
& $\vcenter{\hbox{\tiny 14885}}$ 
& $\vcenter{\hbox{\tiny 4733}}$ 
& $\vcenter{\hbox{\tiny 258}}$ 
\end{tabular}
\end{center}
We observe that Numbers~17 and~36 from Theorem~\ref{thm:terminalWHf-new}
both realise the numerical data from ID~3508 in the Graded Ring Database
but the general members of the respective families are non-isomorphic.
\end{remark}

\begin{remark}
For Numbers~35 and~42 from Theorem~\ref{thm:terminalWHf-new}
the linear system $\vert -K_X \vert$ is empty.
In particular these Fano threefolds $X$ do not
admit an elephant, that means a member of
$\vert -K_X \vert$ with at most canonical singularities.
There appear to be only few known examples for this
phenomenon,
compare~\cite[16.7]{IF} and~\cite[Sec.~4]{San}.
\end{remark}

\bigskip

\noindent
We would like to thank Victor Batyrev for
stimulating seminar talks and discussions
drawing our attention to the class of varieties
defined by non-degenerate systems of Laurent
polynomials.

\tableofcontents

\section{Background on toric varieties}
\label{sec:tvbackground}

In this section, we gather the necessary
concepts and results from toric geometry
and thereby fix our notation. 
We briefly touch some of the fundamental 
definitions but nevertheless assume the reader 
to be familiar with the foundations of the 
theory of toric varieties.
We refer to~\cite{Dan,Ful,CoLiSc} as introductory 
texts.

Our ground field $\KK$ is algebraically closed 
and of characteristic zero.
We write $\TT^n$ for the standard $n$-torus,
that means the $n$-fold direct product of
the multiplicative group $\KK^*$.
By a torus we mean an affine algebraic group
$\TT$ isomorphic to some $\TT^n$.
A toric variety is a normal algebraic
variety $Z$ containing a torus~$\TT$ as a
dense open subset such that the
multiplication on $\TT$ extends to an
action of $\TT$ on $Z$.

Toric varieties are in covariant categorical 
equivalence with lattice fans.
In this context, a lattice is a free 
$\ZZ$-module of finite dimension. 
Moreover, a quasifan (a fan) in a lattice~$N$ 
is a finite collection $\Sigma$ of (pointed)
convex polyhedral cones $\sigma$ in the
rational vector space $N_\QQ = \QQ \otimes_\ZZ N$
such that given $\sigma \in \Sigma$, we have 
$\tau \in \Sigma$ for all faces
$\tau \preccurlyeq \sigma$ and for any two
$\sigma, \sigma' \in \Sigma$, the intersection
$\sigma \cap \sigma'$ is a face of both,
$\sigma$ and $\sigma'$. 
The toric variety $Z$ and its acting torus $\TT$ 
associated with a fan $\Sigma$ in $N$ are 
constructed as follows: 
$$ 
\TT 
\ := \ 
\Spec \, \KK[M],
\qquad
Z
\ := \ 
\bigcup_{\sigma \in \Sigma} Z_\sigma,
\qquad
Z_\sigma \ := \ \Spec \, \KK[\sigma^\vee \cap M], 
$$
where $M$ is the dual lattice of $N$ and
$\sigma^\vee \subseteq M_\QQ$ is the dual cone 
of $\sigma \subseteq N_\QQ$.
The inclusion $\TT \subseteq Z$ of the 
acting torus is given by the inclusion
of semigroup algebras arising from the 
inclusions 
$\sigma^\vee \cap M  \subseteq M$ 
of additive semigroups.
In practice, we will mostly deal with 
$N = \ZZ^n = M$, where $\ZZ^n$ is identified
with its dual via the standard bilinear 
form $\bangle{u,v} = u_1v_1 + \ldots + u_nv_n$.
In this setting, we have $N_\QQ = \QQ^n = M_\QQ$.
Moreover, given a lattice homomorphism 
$F \colon N \to N'$, we write as well
$F \colon N_\QQ \to N'_\QQ$ for  
the associated vector space homomorphism.

We briefly recall Cox's quotient construction
$p \colon \hat Z \to Z$ of a toric variety~$Z$ 
given by a fan $\Sigma$ in $\ZZ^n$
from~\cite{Cox}.
We denote by $v_1, \ldots, v_r \in \ZZ^n$  
the primitive generators of $\Sigma$, 
that means the shortest
non-zero integral vectors of the rays
$\varrho_1, \ldots, \varrho_r \in \Sigma$.
We will always assume that
$v_1, \ldots, v_r$ span $\QQ^n$ as a
vector space; geometrically
this means that $Z$ has no torus factor.
By $D_i \subseteq Z$ we denote the toric
prime divisor corresponding to
$\varrho_i \in \Sigma$.
Throughout the article, we will make free 
use of the notation introduced around 
Cox's quotient presentation.

\begin{construction}
\label{constr:toric-cox}
Let $\Sigma$ be a fan in $\ZZ^n$ and
$Z$ the associated toric variety.
Consider the linear map
$P \colon \ZZ^r \to \ZZ^n$
sending the $i$-th canonical
basis vector $e_i \in \ZZ^r$
to the $i$-th primitive generator
$v_i \in \ZZ^n$ of $\Sigma$,
denote by $\delta = \QQ^r_{\ge 0}$
the positive orthant and define a
fan $\hat \Sigma$ in~$\ZZ^r$ by
$$
\hat \Sigma
\ := \
\{\delta_0 \preccurlyeq \delta; \
P(\delta_0) \subseteq \sigma 
\text{ for some } \sigma \in \Sigma\}.
$$
As $\hat \Sigma$ consists of faces
of the orthant $\delta$, the toric
variety $\hat Z$ defined by $\hat \Sigma$
is an open $\TT^r$-invariant subset
of $\bar Z = \KK^r$.
We also regard the linear map
$P \colon \ZZ^r \to \ZZ^n$
as an $n \times r$ matrix $P = (p_{ij})$
and then speak about the generator
matrix of $\Sigma$.
The generator matrix $P$ defines a homomorphism
of tori:
$$
p \colon \TT^r \ \to \ \TT^n,
\qquad
t \ \mapsto \
(t_1^{p_{11}} \cdots t_r^{p_{1r}}, \ldots, t_1^{p_{n1}} \cdots t_r^{p_{nr}}).
$$
This homomorphism extends to a morphism
$p \colon \hat Z \to Z$ of toric varieties, which
in fact is a good quotient for the action
of the quasitorus $H = \ker(p)$ on $\hat Z$.
Let $P^*$ be the transpose of $P$,
set $K := \ZZ^r/\im(P^*)$ and let
$Q \colon \ZZ^r \to K$ be the projection.
Then $\deg(T_i) := Q(e_i) \in K$
defines a $K$-graded polynomial ring
$$
\mathcal{R}(Z)
\ := \
\bigoplus_{w \in K} \mathcal{R}(Z)_w
\ := \
\bigoplus_{w \in K} \KK[T_1,\ldots,T_r]_w
\ = \
\KK[T_1,\ldots,T_r].
$$
There is an isomorphism $K \to \Cl(Z)$ from
the grading group $K$ onto the divisor class
group $\Cl(Z)$ sending $Q(e_i) \in K$ to the
class $[D_i] \in \Cl(Z)$ of the toric prime 
divisor $D_i \subseteq Z$
defined by the ray $\varrho_i$ through $v_i$.
Moreover, the $K$-graded polynomial ring
$\mathcal{R}(Z)$ is the Cox ring of $Z$;
see~\cite[Sec.~2.1.3]{ArDeHaLa}.
\end{construction}

We now explain the correspondence between 
effective Weil divisors on a toric variety~$Z$ 
and the $K$-homogeneous elements in the 
polynomial ring $\mathcal{R}(Z)$.
For any variety $X$, we denote by
$X_{\reg} \subseteq X$ the open subset of
its smooth points and by $\WDiv(X)$ its
group of Weil divisors.
We need the following pull back
construction of Weil divisors with
respect to morphisms
$\varphi \colon X \to Y$:
Given a Weil divisor $D$ having
$\varphi(X)$ not inside its support,
restrict $D$ to a Cartier divisor
on $Y_{\reg}$, apply the usual pull back
and turn the result into a Weil divisor
on $X$ by replacing its prime components
with their closures in $X$.

\begin{definition}
\label{def:descrpol}
Consider a toric variety $Z$ and its 
quotient presentation $p \colon \hat Z \to Z$.
A \emph{describing polynomial}
of an effective divisor $D \in \WDiv(Z)$
is a $K$-homogeneous polynomial
$g \in \mathcal{R}(Z)$ with
$\ddiv(g) = p^* D \in \WDiv(\bar Z)$.
\end{definition}

\begin{example}
An effective toric divisor $a_1D_1+ \ldots + a_rD_r$ on~$Z$
has the monomial $T_1^{a_1} \cdots T_r^{a_r} \in \mathcal{R}(Z)$
as a describing polynomial.
Moreover, in $K = \Cl(X)$, we have
$$ 
\deg(T_1^{a_1} \cdots T_r^{a_r})
\ = \ 
Q(a_1, \ldots, a_r)
\ = \ 
[a_1D_1 + \ldots + a_rD_r].
$$ 
\end{example}

\goodbreak

We list the basic properties of describing 
polynomials, which in fact hold in 
the much more general framework of Cox rings;
see~\cite[Prop.~1.6.2.1 and Cor~1.6.4.6]{ArDeHaLa}.

\begin{proposition}
\label{prop:descr-pol}
Let $Z$ be a toric variety with quotient
presentation $p \colon \hat Z \to Z$ as in
Construction~\ref{constr:toric-cox}
and let $D$ be any effective
Weil divisor on $Z$.
\begin{enumerate}
\item
There exist describing polynomials
for $D$ and any two of them differ by
a non-zero scalar factor.
\item
If $g$ is a describing polynomial for
$D$, then, identifying $K$ and $\Cl(Z)$
under the isomorphism presented 
in Construction~\ref{constr:toric-cox},
we have 
$$
\qquad
p_*(\ddiv(g))
\ = \ 
D,
\qquad
\deg(g)
\ = \
[D]
\ \in \
\Cl(Z)
\ = \
K.
$$
\item 
For every $K$-homogeneous element $g \in \mathcal{R}(Z)$, 
the divisor $p_*(\ddiv(g))$ is effective and has 
$g$ as a describing polynomial. 
\end{enumerate}  
\end{proposition}

Let us see how base points of effective divisors 
on toric varieties are detected in terms of fans and 
homogeneous polynomials.
Recall that each cone $\sigma \in \Sigma$ 
defines a distinguished point $z_\sigma \in Z$
and the toric variety $Z$ is the disjoint union 
over the orbits $\TT^n \cdot z_\sigma$, where 
$\sigma \in \Sigma$.

\begin{proposition}
\label{prop:basepointfree}
Let $Z$ be the toric variety
arising from a fan $\Sigma$
in~$\ZZ^n$ and~$D$ an effective 
Weil divisor on $Z$.
Then the base locus of $D$ is
$\TT^n$-invariant.
Moreover, 
a point $z_\sigma \in Z$ 
is not a base point of~$D$ if and only if
$D$ is linearly equivalent to
an effective toric divisor
$a_1 D_1 + \ldots + a_rD_r$ 
with $a_i = 0$ whenever
$v_i \in \sigma$. 
\end{proposition}

In the later construction and study of non-degenerate 
subvarieties of toric varieties, we make essential 
use of the normal fan of a lattice polytope
and the correspondence between polytopes and divisors
for toric varieties.
Let us briefly recall the necessary background and notation.

\begin{reminder}
\label{rem:normalfan}
Consider a polytope $B \subseteq \QQ^n$.
We write $B' \preccurlyeq B$ for the 
faces of~$B$.
One obtains a quasifan $\Sigma(B)$
in $\ZZ^n$ by
$$
\Sigma(B)
\ := \ 
\{\sigma(B'); \ B' \preceq B\},
\qquad
\sigma(B') 
\ := \ 
\cone(u - u'; \ u \in B, \ u' \in B')^\vee,
$$
called the \emph{normal fan} of~$B$.
The assignment $B' \mapsto \sigma(B')$
sets up an 
inclusion-reversing bijection 
between the faces of $B$ and the cones 
of $\Sigma(B)$. 
\end{reminder}

Note the slight abuse of notation:
the normal fan $\Sigma(B)$ is a fan in
the strict sense only if the polytope $B$
is of full dimension~$n$, otherwise
$\Sigma(B)$ is a quasifan.
Given quasifans $\Sigma$ and $\Sigma'$ in $\ZZ^n$,
we speak of a \emph{refinement} $\Sigma' \to \Sigma$
if $\Sigma$ and $\Sigma'$ have the same
support and every cone of $\Sigma'$ is
contained in a cone of $\Sigma$.

\begin{reminder}
\label{rem:minkowskisum}
Let $B = B_1 + \ldots +  B_s$ be the
Minkowski sum of polytopes
$B_1, \ldots, B_s \subseteq \QQ^n$.
Each face $B' \preccurlyeq B$
has a unique presentation
$$
B' \ = \ B_1' + \ldots + B_s',
\qquad\qquad
B_1' \preccurlyeq  B_1,
\ldots,
B_s' \preccurlyeq  B_s.
$$
The normal fan $\Sigma(B)$ of $B$ is
the coarsest common refinement of the
normal fans $\Sigma(B_i)$ of the $B_i$.
The cones of $\Sigma(B)$ are
given
as
$$
\sigma(B') 
\ = \
\sigma(B_1') \cap \ldots \cap \sigma(B_s'),
$$
where $B' \preccurlyeq B$ and
$B' = B_1' + \ldots + B_s'$ is the above
decomposition. In particular,  
$\sigma(B_i') \in \Sigma(B_i)$ is the minimal
cone containing $\sigma(B') \in \Sigma(B')$. 
\end{reminder}

\begin{reminder}
\label{rem:divisorsAndPolytopes}
Let $B \subseteq \QQ^n$ be an $n$-dimensional
polytope with integral vertices and let
$\Sigma$ be any complete fan in $\ZZ^n$ with
generator matrix $P = [v_1, \ldots,v_r]$.
Define a vector $a \in \ZZ^r$ by
$$
a \ := \ (a_1,\ldots,a_r),
\qquad \qquad 
a_i
\ := \
-\min_{u \in B} \bangle{u, v_i}.
$$
Observe that the $a_i$ are indeed integers,
because $B$ has integral vertices. 
For $u \in B$ set $a(u) := P^*u+a$
and let $B(u) \preccurlyeq B$ be the
minimal face containing $u$. 
Then the entries of the vector $a(u) \in \QQ^r$ 
satisfy
$$ 
a(u)_i \ge 0, \text{ for } i = 1, \ldots, r,
\qquad \qquad 
a(u)_i = 0
\ \Leftrightarrow \ 
v_i \in \sigma(B(u)).
$$
\end{reminder}

\begin{proposition}
\label{prop:normfanbpfample}
Let $B \subseteq \QQ^n$ be a
lattice polytope and $\Sigma$ any complete
fan in $\ZZ^n$ with generator matrix
$P = [v_1,\ldots,v_r]$.
With $a \in \ZZ^r$
from Reminder~\ref{rem:divisorsAndPolytopes},
we define a divisor on the toric variety
$Z$ arising from $\Sigma$ by
$$
D
\ := \
a_1 D_1 + \ldots + a_r D_r
\ \in \
\WDiv(Z).
$$
Moreover, for every vector $u \in B \cap \ZZ^n$,
we have $a(u) \in \ZZ^{r}$ as in 
Reminder~\ref{rem:divisorsAndPolytopes}
and obtain effective
divisors $D(u)$ on $Z$, all of
the same class as $D$ by   
$$
D(u)
\ := \
a(u)_1 D_1 + \ldots + a(u)_r D_r
\ \in \
\WDiv(Z).
$$
If $\Sigma$ refines the normal fan $\Sigma(B)$,
then $D$ and all $D(u)$ are base point free.
If $\Sigma$ equals the normal fan $\Sigma(B)$,
then the divisors $D$ and $D(u)$ are even ample.
\end{proposition}

\section{Laurent systems and their Newton polytopes}
\label{sec:laurent-systems}

We consider systems $F$ of Laurent
polynomials in $n$ variables.
Any such system~$F$ defines a
Newton polytope~$B$ in $\QQ^n$. 
The objects of interest are
completions $X \subseteq Z$
of the zero set $V(F) \subseteq \TT^n$
in the toric varieties $Z$ 
associated with refinements of
the normal fan of $B$.
In Proposition~\ref{prop:differential},
we interpret Khovanskii's
non-degeneracy condition~\cite{Kh}
in terms of Cox's quotient
presentation of $Z$.
Theorem~\ref{thm:nonde2ci} gathers
complete intersection properties
of the embedded varieties
$X \subseteq Z$ given by non-degenerate
systems of Laurent polynomials.

We begin with recalling the
basic notions around Laurent
polynomials and Newton polytopes.
Laurent polynomials are the
elements of the Laurent polynomial
algebra for which we will use the
short notation
$$
\LP(n)
\ := \
\KK[T_1^{\pm 1}, \ldots, T_n^{\pm 1}].
$$

\begin{definition}
Take any Laurent polynomial
$f = \sum_{\nu \in \ZZ^n} \alpha_\nu T^\nu \in \LP(n)$.
The \emph{Newton polytope}
of $f$ is 
$$
B(f)
\ := \
\conv(\nu \in \ZZ^n; \ \alpha_\nu \ne 0)
\ \subseteq \
\QQ^n.
$$
Given a face $B \preccurlyeq B(f)$ of the
Newton polytope, the associated 
\emph{face polynomial} is defined as 
$$
f_B
\ = \
\sum_{\nu \in B \cap \ZZ^n} \alpha_\nu T^\nu
\ \in \
\LP(n).
$$
\end{definition}

\begin{construction}
\label{constr:homogenization}
Consider a Laurent polynomial
$f \in \LP(n)$
and a fan $\Sigma$ in~$\ZZ^n$.
The pullback of~$f$ with respect
to the homomorphism $p \colon \TT^r \to \TT^n$
defined by the generator matrix $P = (p_{ij})$ of
$\Sigma$ has a unique presentation
as
$$
p^*f(T_1, \ldots, T_r)
\ = \
f(T_1^{p_{11}} \cdots T_r^{p_{1r}}, \ldots, T_1^{p_{n1}} \cdots T_r^{p_{nr}})
\ = \
T^{\nu}g(T_1, \ldots, T_r)
$$
with a Laurent monomial
$T^{\nu} = T_1^{\nu_1} \cdots T_r^{\nu_r} \in \LP(r)$
and a $K$-homo\-geneous polynomial
$g \in \KK[T_1,\ldots,T_r]$ being
coprime to each of the variables
$T_1, \ldots, T_r$.
We call~$g$ the \emph{$\Sigma$-homogenization}
of $f$.
\end{construction}

\begin{lemma}
\label{lem:descrPolyOfHypersurfaces}
Consider a Laurent polynomial $f \in \LP(n)$
with Newton polytope~$B(f)$
and a fan $\Sigma$ in~$\ZZ^n$ with
generator matrix $P:=[v_1, \ldots, v_r]$
and associated toric variety $Z$.
Let $a := (a_1, \ldots, a_r)$
be as in Reminder~\ref{rem:divisorsAndPolytopes}
and $D \in \WDiv(Z)$ the push forward
of $\ddiv(f) \in \WDiv(\TT^n)$.
\begin{enumerate}
\item 
The $\Sigma$-homogenization $g$ of $f$
is a describing polynomial of $D$
and with the homomorphism
$p \colon \TT^r \rightarrow \TT^n$ 
given by $P$, we have
$$
\qquad
g
\ = \
T^a p^*f
\ \in \
\mathcal{R}(Z), 
\qquad
T^a 
\ :=\
T^{a_1} \cdots T^{a_r}.
$$
\item
The Newton polytope of $g$ equals the 
image of the Newton polytope of $f$ under 
the injection $\QQ^n \to \QQ^r$ sending
$u$ to $a(u) := P^*u +a$,
in other words
$$ 
\qquad  
B(g) 
\ = \ 
P^*B(f) + a
\ =  \
\{a(u); \; u \in B(f)\}.
$$
\item
Consider a face $B \preccurlyeq B(f)$ and the 
associated face polynomial $f_B$. 
Then the corresponding face $P^*B + a \preccurlyeq B(g)$ 
has the face polynomial
$$ 
\qquad
g_{P^*B + a} \ = \ g(\tilde T_1, \ldots, \tilde T_r),
\qquad
\tilde T_i 
\ := \ 
\begin{cases}
0 & v_i \in \sigma(B),
\\
T_i & v_i \not\in \sigma(B).
\end{cases}
$$
Moreover, for each monomial $T^\nu$ of $g - g_{P^*B + a}$
there is a proper face $\sigma \prec \sigma(B)$
such that every variable $T_i$ with
$v_i \in \sigma(B) \setminus \sigma$ 
divides $T^\nu$. 
\item
The degree $\deg(g) \in K$ of the $\Sigma$-homogenization
$g$ of $f$ and the divisor class $[D] \in \Cl(Z)$ of
$D \in \WDiv(Z)$ are given by
$$
\qquad\qquad
\deg(g)
\ = \
Q(a)
\ = \
[a_1D_1 + \ldots + a_rD_r]
\ = \
[D].
$$
\item
If $\Sigma$ is a refinement of the normal fan of $B(f)$,
then the divisor $D \in \WDiv(Z)$ is base point free on $Z$.
\end{enumerate}
\end{lemma}

\begin{proof}
Assertions (i) to (iii) are direct consequences of
Reminder~\ref{rem:divisorsAndPolytopes}.
Assertion~(iv) is clear by Proposition~\ref{prop:descr-pol}
and~(v) follows from Proposition~\ref{prop:normfanbpfample}.
\end{proof}

Here are the basic notions around
systems of Laurent polynomials;
observe that item~(iii) is precisely
Khovanskii's non-degeneracy condition
stated in~\cite[Sec.~2.1]{Kh}.

\begin{definition}
\label{def:laurent-system}
Let $f_1, \ldots, f_s \in \LP(n)$
be Laurent polynomials with 
Newton polytopes $B_j := B(f_j) \subseteq \QQ^n$.
\begin{enumerate}
\item
We speak of $F = (f_1, \ldots, f_s)$ as a \emph{system}
in $\LP(n)$ and define the Newton
polytope of $F$ to be the Minkowski sum
$$
B \ := \ B(F) \ = \ B_1 + \ldots + B_s \ \subseteq \ \QQ^n.
$$ 
\item
The \emph{face system} $F'$ of $F$ associated with a face
$B' \preccurlyeq B$ of the Newton polytope is
the Laurent system
$$
F' \ = \ F_{B'} \ = \ (f_1', \ldots, f_s'),
$$
where $f_j' = f_{B_j'}$ are the face polynomials
associated with the faces $B_j' \preccurlyeq B_j$
from the presentation $B' = B_1' + \ldots + B_s'$.
\item
We call $F$ \emph{non-degenerate} if for every
face $B' \preccurlyeq B$, the differential
$\mathcal{D} F'(z)$ is of rank $s$ for all
$z \in V(F') \subseteq \TT^n$.
\item
Let $\Sigma$ be a fan in $\ZZ^n$. The
\emph{$\Sigma$-homogenization} of
$F = (f_1, \ldots, f_s)$ is the system
$G  = (g_1, \ldots, g_s)$,
where $g_j$ is the $\Sigma$-homogenization of $f_j$.
\item
By an \emph{$F$-fan} we mean a fan $\Sigma$ in $\ZZ^n$
that refines the normal fan $\Sigma(B)$
of the Newton polytope $B$ of~$F$.  
\end{enumerate}
\end{definition}

Note that Khovanskii's non-degeneracy
Condition~\ref{def:laurent-system}~(iii)
is fullfilled for suitably general choices of~$F$.
Even more, it is a concrete condition in the
sense that for every explicitly given Laurent system~$F$,
we can explicitly check non-degeneracy.

\begin{construction}
\label{constr:laurent-system-zeros}
Consider a system $F = (f_1, \ldots, f_s)$
in $\LP(n)$, a fan $\Sigma$ in~$\ZZ^n$ and 
the $\Sigma$-homogenization $G$ of $F$.
Define subvarieties
$$
\bar X
:= 
V(G)
:= 
V(g_1, \ldots , g_s)
\subseteq 
\bar Z,
\qquad
X
:=
\overline{V(f_1)} \cap \ldots \cap \overline{V(f_s)}
\subseteq
Z,
$$
where $Z$ is the toric variety associated
with $\Sigma$ and $\bar Z = \KK^r$.
The quotient presentation $p \colon \hat Z \to Z$
gives rise to a commutative diagram
$$
\xymatrix{
{\hat X}
\ar@{}[r]|\subseteq
\ar[d]^p_{\quot H}
&
{\hat Z}
\ar[d]^{\quot H}_p
\\
X
\ar@{}[r]|\subseteq
&
Z
}
$$
where $\hat X := \bar X \cap \hat Z \subseteq \bar Z$
as well as $X \subseteq Z$ are closed subvarieties
and $p \colon \hat X \to X$ is a good quotient for
the induced $H$-action on $\hat X$.
In particular, $X = p(\hat X)$.
\end{construction}

In our study of $\bar X$, $\hat X$ and $X$,
the decompositions induced from the respective
ambient toric orbit decompositions
will play an importante role.
We work with distinguished points
$z_\sigma \in Z$.
In terms of Cox's quotient presentation,
$z_\sigma \in Z$ becomes explicit as  
$z_{\sigma} = p(z_{\hat \sigma})$, where 
$\hat \sigma = \cone (e_i; \ v_i \in \sigma) \in \hat \Sigma$ 
and the coordinates of the distinguished point 
$z_{\hat \sigma} \in \hat Z$ are 
$z_{\hat \sigma, i} = 0$ 
if $v_i \in \sigma$ and $z_{\hat \sigma, i} = 1$
otherwise.

\begin{construction}
\label{constr:pieces}
Consider a system $F = (f_1, \ldots, f_s)$
in $\LP(n)$, a fan $\Sigma$ in~$\ZZ^n$ and 
the $\Sigma$-homogenization $G = (g_1, \ldots, g_s)$
of $F$.
For every cone $\sigma \in \Sigma$ define
$$ 
g_j^{\sigma} \ := \ g_j(T_1^{\sigma}, \ldots, T_r^{\sigma}),
\qquad
T_i^{\sigma} 
\ := \ 
\begin{cases}
0 & v_i \in \sigma,
\\
T_i & v_i \not\in \sigma.
\end{cases}
$$
This gives us a system
$G^\sigma := (g_1^{\sigma}, \ldots, g_s^{\sigma})$
of polynomials in $\KK[T_i; \ v_i \not\in \sigma]$.
In the coordinate subspace 
$\bar Z(\sigma) = V(T_i; \ v_i \in \sigma)$
of~$\KK^r$, we have
$$
\bar X (\sigma)
\ := \
\bar X \cap \bar Z(\sigma)
\ = \
V(G^\sigma)
\ \subseteq \ 
\bar Z(\sigma).
$$
Note that $\bar Z(\sigma)$ equals the
closure of the toric orbit
$\TT^r \cdot z_{\hat \sigma}
\subseteq \KK^r$.
Consider as well the toric orbit
$\TT^n \cdot z_\sigma \subseteq Z$
and define locally closed subsets
$$
\hat X (\sigma)
\ := \
\hat X \cap \TT^r \cdot z_{\hat \sigma}
\ \subseteq \
\hat X,
\qquad
X (\sigma)
\ := \
X \cap \TT^n \cdot z_\sigma
\ \subseteq \
X.
$$
Then we have $X(\sigma) = p(\hat X(\sigma))$
and $X \subseteq Z$ is the disjoint union of
the subsets $X(\sigma)$, where $\sigma \in \Sigma$.
\end{construction}

The key step for our investigation of
varieties $X \subseteq Z$ defined by 
Laurent systems is to interprete the
non-degeneracy condition of a system
$F$ in terms of its
$\Sigma$-homogenization~$G$.

\begin{proposition}
\label{prop:differential}
Let $F = (f_1,\ldots,f_s)$ be a non-degenerate
system in $\LP(n)$ and let~$\Sigma$
be an $F$-fan in $\ZZ^n$.
\begin{enumerate}
\item
The differential $\mathcal{D} G (\hat z)$
of the $\Sigma$-homogenization $G$ of $F$
is of full rank~$s$ at every point
$\hat z \in \hat X$.
\item
For each cone $\sigma \in \Sigma$,
the differential $\mathcal{D} G^\sigma (\hat z)$
of the system $G^\sigma$ is of full rank~$s$ at every
point $\hat z \in \hat X(\sigma)$.
\item
For every $\sigma \in \Sigma$,
the scheme $\hat X(\sigma) := \hat X \cap \TT^r \cdot z_{\hat \sigma}$,
provided it is non-empty,
is a closed subvariety of pure codimension $s$
in $\TT^r \cdot z_{\hat \sigma}$.
\end{enumerate}  
\end{proposition}

\begin{proof}
We care about~(i) and on the way also prove~(ii).
Since $g_1, \ldots, g_s$ are $H$-homogeneous,
the set of points $\hat z \in \hat Z$ with
$\mathcal{D} G (\hat z)$ of rank strictly less
than $s$ is $H$-invariant and closed in~$\hat Z$.
Thus, as $p \colon \hat Z \to Z$ is a good quotient
for the $H$-action, it suffices to show that
for the points $\hat z \in \hat X$
with a closed $H$-orbit in~$\hat Z$, the
differential $\mathcal{D} G (\hat z)$ is
of rank $s$.
That means that we only have to deal with the
points $\hat z \in \hat X \cap \TT^r \cdot z_{\hat \sigma}$,
where $\sigma \in \Sigma$.

So, consider a point $\hat z \in \hat X \cap \TT^r \cdot z_{\hat \sigma}$,
let $\sigma' \in \Sigma(B)$ be the minimal cone with
$\sigma \subseteq \sigma'$ and let $B' \preccurlyeq B$
be the face corresponding to $\sigma' \in \Sigma(B)$.
Then we have the Minkowski decomposition
$$
B' \ = \ B_1' + \ldots + B_s',
\qquad
B_j' \preccurlyeq B_j = B(f_j).
$$
From Reminder~\ref{rem:minkowskisum} we infer
that $\sigma_j' = \sigma(B_j')$
is the minimal cone of the normal fan
$\Sigma(B_j')$ with $\sigma \subseteq \sigma_j'$.
Let $F'$ be the face system of $F$ given by
$B' \subseteq B$.
Define $G' = (g_1' , \ldots, g_s')$, where
$g_j'$ is the face polynomial of $g_j$
defined by
$$
P^*B_j' + a_j \ \preccurlyeq \ P^*B_j + a_j \ = \ B(g_j),
\qquad
g_j \ = \ T^{a_j}p^*f_j.
$$
According to Lemma~\ref{lem:descrPolyOfHypersurfaces}~(iii),
the polynomials
$g_j'$ only depend on the variables $T_i$ 
with $v_i \not\in \sigma(B_j')$.
Moreover, we have 
$$
g_j' \ =\  g_j^\sigma,
\quad j = 1, \ldots, s,
$$
because due to the  minimality of $\sigma_j' = \sigma(B_j')$
each monomial of $g_j - g_j'$ is a multiple
of some~$T_i$ with $v_i \in \sigma$.
Thus, $G' = G^\sigma$.
Using the fact that $\hat z_i = 0$
if and only if $v_i \in \sigma$, 
we observe
$$
g_j^\sigma( \hat z )  = g_j( \hat z )  = 0,
\
j = 1, \ldots, s,
\qquad
\rank \, \mathcal{D}G^\sigma( \hat z ) = s
\ \Rightarrow \
\rank \, \mathcal{D}G( \hat z ) = s.
$$
This reduces the proof of~(i) to showing
that $\mathcal{D}G^\sigma( \hat z )$ is
of full rank~$s$, and the latter
proves~(ii).
Choose $\tilde z \in \TT^r$ 
such that $\tilde z_i = \hat z_i$
for all $i$ with $v_i \not\in \sigma$.
Using again that the polynomials $g_i'$
only depend on $T_i$ with $v_i \not\in \sigma$,
we see
$$ 
g_j^\sigma( \tilde z )  = g_j^\sigma( \hat z )  = 0,
\ j = 1, \ldots, s,
\qquad
\mathcal{D}G^\sigma( \hat z )
= 
\mathcal{D}G^\sigma( \tilde z ).
$$ 
We conclude that $F'(p(\tilde z)) = 0$ holds.
Thus, the non-degeneracy condition on the
Laurent system~$F$ 
ensures that $\mathcal{D}F'(p(\tilde z))$
is of full rank~$s$.
Moreover, we have
$$
\mathcal{D}G^\sigma( \hat z )
\ = \ 
\mathcal{D}G^\sigma( \tilde z )
\ = \ 
(T^{a_1}, \ldots, T^{a_s})(\tilde z)
\cdot
\mathcal{D}F'(p(\tilde z)) \circ  \mathcal{D}p(\tilde z).
$$
Since $T^{a_j}(\tilde z) \ne 0$ holds 
for $j = 1, \ldots, s$ and 
$p \colon \TT^r \to \TT^n$ is a submersion,
we finally obtain that
$\mathcal{D}G^\sigma( \hat z )$ is of full
rank~$s$, which proves~(i) and~(ii).
Assertion~(iii) follows from~(ii)
and the Jacobian criterion for complete
intersections.
\end{proof}

\begin{remark}
Given a system $F$ in $\LP(n)$ and an $F$-fan 
$\Sigma$ in $\ZZ^n$, let $G$ be the $\Sigma$-ho\-mogenization
of $F$. 
The proof of Proposition~\ref{prop:differential}
shows that $F$ is non-degenerate
if and only if all $\mathcal{D} G^\sigma(\hat z)$, 
where $\sigma \in \Sigma$ and $\hat z \in \hat X(\sigma)$,
are of full rank.
\end{remark}

A first application gathers complete
intersection properties for the varieties
defined by a non-degenerate
Laurent system.
Note that the codimension condition 
imposed on $\bar X \setminus \hat X$
in the fourth assertion below
allows computational verification
for explicitly given systems
of Laurent polynomials.

\begin{theorem}
\label{thm:nonde2ci}
Consider a non-degenerate system $F = (f_1,\ldots,f_s)$ 
in $\LP(n)$, an $F$-fan $\Sigma$ in $\ZZ^n$ 
and the $\Sigma$-homogenization $G = (g_1,\ldots,g_s)$
of $F$.
\begin{enumerate}
\item
The variety $\bar X = V(G)$ in $\bar Z = \KK^r$
is a complete intersection of pure dimension
$r-s$ with vanishing ideal
$$
\qquad
I(\bar X)
\ = \
\bangle{g_1, \ldots, g_r}
\ \subseteq \
\KK[T_1, \ldots, T_r].
$$
\item
With the zero sets $V(F) \subseteq \TT^n$
and  $V(G) \subseteq \KK^r$
and the notation of 
Construction~\ref{constr:laurent-system-zeros}, 
we have
$$
\qquad\qquad
\hat X
\ = \
\overline{V(G) \cap \TT^r}
\ \subseteq \
\hat Z,
\qquad\qquad
X
\ = \
\overline{V(F)}
\ \subseteq \
Z.
$$
In particular, the irreducible components 
of $X \subseteq Z$ are the closures of 
the irreducible components of $V(F) \subseteq \TT^n$.
\item
The closed hypersurfaces
$X_j = \overline{V(f_j)} \subseteq Z$,
where $j = 1, \ldots, s$, 
represent~$X$ as a scheme-theoretic
locally complete intersection
$$
\qquad
X \ = \ X_1 \cap \ldots \cap  X_s \ \subseteq \ Z.
$$
\item
If $\bar X \setminus \hat X$ is of codimension 
at least two in $\bar X$, then $\bar X$ is
irreducible and normal and, moreover,
$X$ is irreducible. 
\end{enumerate}
\end{theorem}

\begin{proof}
Assertion~(i) is clear by Proposition~\ref{prop:differential}~(i)
and the Jacobian criterion for complete intersections.
For~(ii), we infer
from Proposition \ref{prop:differential}~(ii)
that, provided it is non-empty, the intersection
$\hat X \cap \TT^r \cdot z_{\hat \sigma}$
is of dimension $r - s - \dim(\hat \sigma)$.
In particular no irreducible component of
$V(G)$ is contained in $\hat X \setminus \TT^r$.
The assertions follow.

We prove~(iii).
Each $f_j$ defines a divisor on $Z$ having
support $X_j$ and according to
Lemma~\ref{lem:descrPolyOfHypersurfaces}~(v)
this divisor is base point free on $Z$.
Thus, for every $\sigma \in \Sigma$,
we find a monomial $h_{\sigma,j}$ of
the same $K$-degree as $g_j$ without
zeroes on the affine chart
$\hat Z_{\hat \sigma} \subseteq \hat Z$
defined by $\hat \sigma \in \hat \Sigma$.
We conclude that the invariant functions
$g_1/h_{\sigma, 1}, \ldots, g_s/h_{\sigma, s}$
generate the vanishing ideal of $X$
on the affine toric chart $Z_\sigma \subseteq Z$.

We turn to~(iv).
Proposition~\ref{prop:differential}
and the assumption that $\bar X \setminus \hat X$
is of codimension at least two in $\bar X$
allow us to apply Serre's criterion
and we obtain that $\bar X$ is normal.
In order to see that $\bar X$ is irreducible,
note that $H$ acts on $\bar Z$ with attractive
fixed point $0 \in \bar Z$.
This implies $0 \in \bar X$,
Hence $\bar X$ is connected and thus,
by normality, irreducible.
\end{proof}

\begin{remark}
If in the setting of Theorem~\ref{thm:nonde2ci},
the dimension of $\bar Z \setminus \hat Z$
is at most $r- s -2$,
for instance if $Z$ is a fake weighted projective
space, then the assumption of Statement~(iv)
is satisfied.
\end{remark}

The statements~(i) and~(iv) of Theorem~\ref{thm:nonde2ci}
extend in the following way to the pieces cut out
from $\bar X$ by the closures of the $\TT^r$-orbits
of $\bar Z = \KK^r$.

\begin{proposition}
\label{prop:piecewiseci}
Consider a non-degenerate system $F = (f_1,\ldots,f_s)$ 
in $\LP(n)$, an $F$-fan $\Sigma$ in $\ZZ^n$,
the $\Sigma$-homogenization $G = (g_1,\ldots,g_s)$
of $F$, a cone $\sigma \in \Sigma$ and 
$$
\bar Z(\sigma) 
\ = \ 
V(T_i; \ v_i \in \sigma),
\qquad
\bar X(\sigma) 
\ = \ \
\bar X \cap \bar Z(\sigma).
$$
If $\bar X(\sigma) \setminus \hat X (\sigma)$ is
of codimension least one in $\bar X(\sigma)$,
then $\bar X(\sigma) = \bar X \cap \bar Z(\sigma)$ 
is a subvariety of pure codimension $s$ 
in $\bar Z(\sigma)$ with vanishing ideal
$$ 
I(\bar X(\sigma)) 
\ = \ 
\bangle{g_1^\sigma, \ldots, g_s^\sigma}
\ \subseteq \ 
\KK[T_i; \ v_i \not\in \sigma].
$$
If $\bar X(\sigma) \setminus \hat X$ is 
of codimension at least two in $\bar X(\sigma)$,
then the variety $\bar X(\sigma)$ is
irreducible and normal.  
\end{proposition}

\begin{proof}
If $\bar X(\sigma) \setminus \hat X (\sigma)$ is
of codimension least one in $\bar X(\sigma)$,
then Proposition~\ref{prop:differential} and
the Jacobian criterion ensure that $\bar X(\sigma)$
is a complete intersection
in $\KK^r$ with the equations $g_j = 0$, $j = 1, \ldots, s$,
and $T_i = 0$, $v_i \in \sigma$.
This gives the first statement.
If $\bar X(\sigma) \setminus \hat X (\sigma)$ is
of codimension at least two in $\bar X(\sigma)$,
then we obtain irreducibility
and normality as in the proof of~(iv)
of Theorem~\ref{thm:nonde2ci}, 
replacing~$\bar X$ with $\bar X(\sigma)$.
\end{proof}

\goodbreak

\section{Non-degenerate toric complete intersections}

We take a closer look at the geometry
of the varieties $X \subseteq Z$ arising
from non-degenerate Laurent systems.
The main statements of the section
are Theorem~\ref{thm:nondegls2smooth},
showing that $X \subseteq Z$ is always
quasismooth and Theorem~\ref{thm:nondeg2trop}
giving details on how~$X$ sits inside $Z$.
Using these, we can prove
Theorem~\ref{thm:main-1} which
describes the anticanonical complex.
First we give a name to
our varieties $X \subseteq Z$,
motivated by Theorem~\ref{thm:nondeg2trop}.
Finally, we see that for a general
choice of the defining Laurent system
and an easy-to-check assumption on the ambient
toric variety $Z$, we obtain divisor class
group and Cox ring of $X$ for free, see
Corollary~\ref{cor:general2SrRa}.

\begin{definition}
By a \emph{non-degenerate toric complete
intersection} we mean a variety
$X \subseteq Z$ defined by a non-degenerate
system $F$ in $\LP(n)$ and an $F$-fan
$\Sigma$ in~$\ZZ^n$.
\end{definition}

An immediate but important property
of non-degenerate toric complete intersections
is quasismoothness; see also~\cite{ACG}
for further results in this direction.
The second statement in the theorem below
is Khovanskii's resolution of
singularities~\cite[Thm.~2.2]{Kh}.
Observe that our proof works without
any ingredients from the theory
of holomorphic functions.

\begin{theorem}
\label{thm:nondegls2smooth}
Let $F$ be a non-degenerate system
in $\LP(n)$
and~$\Sigma$ an $F$-fan in~$\ZZ^n$.
Then the variety $X$ is normal
and quasismooth 
in the sense that $\hat X$ is smooth.
Moreover, $X \cap Z_\reg \subseteq X_\reg$.
In particular, if~$Z$ is smooth,
then $X$ is smooth.
\end{theorem}

\begin{proof}
By Proposition~\ref{prop:differential}~(i),
the variety $\hat X$ is smooth.
As smooth varieties are normal and 
the good quotient $p \colon \hat X \to X$
preserves normality, we see that $X$ is normal.
Moreover, if $Z$ is smooth, then
the quasitorus $H = \ker(p)$ acts freely
on~$p^{-1}(Z_\reg)$,
hence on $\hat X \cap p^{-1}(Z_\reg)$
and thus the quotient
map $p \colon \hat X \to X$ preserves
smoothness over $X \cap Z_\reg$.
\end{proof}

The next aim is to provide details
on the position of $X$ inside the
toric variety~$Z$.
The considerations elaborate the
transversality statement on
$X$ and the torus orbits of~$Z$
made in~\cite{Kh} for the smooth
case.

\begin{definition}
Let $Z$ be the toric variety arising from a fan 
$\Sigma$ in $\ZZ^n$.
Given a closed subvariety $X \subseteq Z$, 
we set
$$
\Sigma_X 
\ := \ 
\{\sigma \in \Sigma; \ X(\sigma) \ne \emptyset\},
\qquad\quad
X(\sigma)
\ = \
X \cap \TT^n \cdot z_\sigma.
$$
\end{definition}

\begin{theorem}
\label{thm:nondeg2trop}
Consider a non-degenerate system
$F = (f_1, \ldots, f_s)$ in~$\mathrm{LP}(n)$,
an $F$-fan~$\Sigma$ in $\ZZ^n$ and 
the associated toric complete 
intersection~$X \subseteq Z$.
\begin{enumerate}
\item
For every $\sigma \in \Sigma_X$, the scheme
$X(\sigma) \cap \TT^n \cdot z_\sigma$
is a closed subvariety of pure codimension~$s$ in 
$\TT^n \cdot z_\sigma$.
\item
The subset $\Sigma_X \subseteq \Sigma$ is 
a subfan and the subset 
$Z_X := \TT^n \cdot X \subseteq Z$
is an open toric subvariety.
\item
All maximal cones of $\Sigma_X$ are of 
dimenson $n-s$ and the support  of $\Sigma_X$
equals the tropical variety 
of $V(F) \subseteq \TT^n$.
\end{enumerate}
\end{theorem}

\begin{proof}
We prove~(i). 
Given a cone $\sigma \in \Sigma_X$
consider $\hat \sigma \in \hat \Sigma$
and the corresponding affine toric 
charts and the restricted quotient
map:
$$ 
\xymatrix{
{\bar X \cap \hat Z_{\hat \sigma}}
\ar@{}[r]|=
&
{\hat X_{\hat \sigma}}
\ar@{}[r]|\subseteq
\ar[d]_p
&
{\hat Z_{\hat \sigma}}
\ar@{}[r]|=
\ar[d]^p
&
p^{-1}(Z_\sigma)
\\
X \cap Z_\sigma
\ar@{}[r]|=
&
X_\sigma
\ar@{}[r]|\subseteq
&
Z_\sigma
&
}
$$
From Proposition~\ref{prop:differential}
we infer that
$\hat X(\hat \sigma) = \TT^r \cdot z_{\hat \sigma} \cap \hat X$
is a reduced subscheme of pure 
codimension~$s$ in 
$\TT^r \cdot z_{\hat \sigma}$.
The involved vanishing ideals
on $Z_\sigma$ and $\hat Z_{\hat \sigma}$
satisfy
$$ 
I(X_\sigma) + I(\TT^n \cdot z_{\sigma})
\ = \ 
I(\hat X_{\hat \sigma})^H + I(\TT^r \cdot z_{\hat \sigma})^H
\ = \ 
\left(I(\hat X_{\hat \sigma}) + I(\TT^r \cdot z_{\hat \sigma})\right)^H.
$$
We conclude that the left hand side ideal
is radical.
In order to see that $X(\sigma)$ is of codimension $s$ in
$\TT^n \cdot z_\sigma$, look at the restriction
$$
p \colon \TT^r \cdot z_{\hat \sigma} \ \to \ \TT^n \cdot z_{\sigma}.
$$
This is a geometric quotient for the $H$-action,
it maps $\hat X(\hat \sigma)$ onto $X(\sigma)$
and, as~$\hat X(\hat \sigma)$ is $H$-invariant,
it preserves codimensions.

We prove~(ii) and~(iii).
First note that, due to~(i),
for any $\sigma \in \Sigma_X$ 
we have $\dim(\sigma) \le n - s$.
We compare $\Sigma_X$ with $\trop(X)$.
Tevelev's criterion~\cite{Tev}
tells us that a cone $\sigma \in \Sigma$
belongs to $\Sigma_X$ if and only if 
$\sigma^\circ \cap \trop(X) \ne \emptyset$
holds.
As $\Sigma$ is complete, 
we conclude that $\trop(X)$ is covered
by the cones of $\Sigma_X$.

We show that the support of every cone
of $\Sigma_X$ is contained in $\trop(X)$.
The tropical structure theorem
provides us with a balanced fan
structure~$\Delta$ on~$\trop(X)$
such that all maximal cones are 
of dimension $n-s$; see~\cite[Thm.~3.3.6]{MLSt}.
Together with Tevelev's criterion,
the latter yields that all maximal
cones of $\Sigma_X$ are of dimension
$n-s$.
The balancy condition implies that
every cone $\delta_0 \in \Delta$
of dimension $n-s-1$ is a facet
of at least two maximal cones of
$\Delta$.
We conclude that every cone
$\sigma \in \Sigma_X$
of dimension $n-s$ must be
covered by maximal cones of
$\Delta$.

Knowing that $\trop(X)$ is precisely
the union of the cones of $\Sigma_X$,
we directly see that $\Sigma_X$ is a
fan: Given $\sigma \in \Sigma_X$,
every face $\tau \preccurlyeq \sigma$
is contained in $\trop(X)$.
In particular, $\tau^\circ$ intersects
$\trop(X)$.
Using once more Tevelev's criterion,
we obtain $\tau \in \Sigma_X$.
\end{proof}

\begin{corollary}
\label{cor:SigmaXcrit}
Let $X \subseteq Z$ be a non-degenerate toric complete 
intersection given by
$F = (f_1, \ldots, f_s)$ in~$\mathrm{LP}(n)$ and
a simplicial $F$-fan $\Sigma$.
If $\bar X \setminus \hat X$ is of dimension strictly less than
$r-n$, then we have 
$$ 
\Sigma_X
\  =  \
\{\sigma \in \Sigma; \ \dim(\sigma) \le n-s\}.
$$
\end{corollary}

\begin{proof}
Assume that $\sigma \in \Sigma$ is of dimension $n-s$
but does not belong to $\Sigma_X$. 
Then $X(\sigma) = \emptyset$ and hence 
$\hat X(\sigma) = \emptyset$. 
This implies
$$ 
\bar X(\sigma) 
\ = \ 
V(g_1,\ldots,g_s) \cap V(T_i; \ v_i \in \sigma) 
\ = \ 
\bigcup_{\hat \sigma \prec \tau} \bar X \cap \TT^r \cdot z_{\tau}.  
$$
As $\Sigma$ is simplicial, $P$ defines a bijection from 
$\hat \Sigma$ onto $\Sigma$. Moreover, $\hat \sigma$ 
and $\sigma$ both have $n-s$ rays and 
we can estimate the dimension of 
$\bar X(\sigma)$ as
$$ 
\dim(\bar X(\sigma))
\ \ge \ 
r - (n-s) 
\ = \  
r-n.
$$
Due to $\dim(\bar X \setminus \hat X) < r-n$, we have 
$\bar X \cap \TT^r \cdot z_{\tau} \subseteq \hat X$ 
for some $\hat \sigma \prec \tau \in \hat \Sigma$.
Thus, $\sigma$ is a proper face of $P(\tau) \in \Sigma_X$.
This contradicts to Theorem~\ref{thm:nondeg2trop}~(iii).
\end{proof}

\begin{example}
\label{ex:kstarsurf}
Let $f = S_1 + S_2 + 1 \in \KK[S_1,S_2,S_3]$ and $\Sigma$
the fan in $\ZZ^3$ given via its generator matrix 
$P = [v_1,\ldots,v_5]$ and maximal cones 
$\sigma_{ijk} = \cone(v_i,v_j,v_k)$:
$$ 
P  
\ = \ 
\left[
\begin{array}{rrrrr}
-2 & 2 & 0 & 0 & 0
\\
-2 & 0 & 2 & 0 & 0  
\\
1  & 1 & 1 & 1 & -1
\end{array}
\right],
\qquad
\Sigma^{\max} 
\ = \ 
\{\sigma_{124},\sigma_{134},\sigma_{234},\sigma_{125}, \sigma_{135},\sigma_{235}\}.
$$
Then $f$ is non-degenerate in $\LP(3)$
and $\Sigma$ is an $F$-fan.
Thus, we obtain a nondegenerate toric hypersurface
$X \in Z$.
The $\Sigma$-homogenization of $f$ is 
$$
g \ = \ T_1^2 +T_2^2 +T_3^2.
$$
The minimal ambient toric variety  
$Z_X \subseteq Z$ is the open toric subvariety 
given by the fan $\Sigma_X$ with the maximal 
cones $\sigma_{ij} = \cone(v_i,v_j)$ given 
as follows
$$ 
\Sigma_X^{\max} 
\ = \ 
\{\sigma_{14},\sigma_{15},\sigma_{24},\sigma_{25}, \sigma_{34},\sigma_{35}\}.
$$
In particular, the fan $\Sigma_X$ is a proper subset of the set 
of all cones of dimension at most two of the fan $\Sigma$. 
%Note that $\trop(X)$ equals the union of $\cone(e_1,\pm e_3)$, 
%$\cone(e_2,\pm e_3)$, and $\cone(-e_1-e2,\pm e_3)$.
\end{example}

\begin{remark}
The variety $X$ from Example~\ref{ex:kstarsurf} is a 
rational $\KK^*$-surface as constructed in~\cite[Sec.~5.4]{ArDeHaLa}.
More generally, every weakly tropical general arrangement 
variety in the sense of~\cite[Sec.~5]{HaHiWr} is an example of 
a non-degenerate complete toric intersection.
\end{remark}

We approach the proof
of Theorem~\ref{thm:main-1}.
The following pull back construction relates 
divisors of $Z$ to divisors on $X$.

\begin{remark}
\label{rem:tci-weilpull}
Let $X \subseteq Z$ be an irreducible non-degenerate
toric complete intersection.
Denote by 
$\imath \colon X \cap Z_\reg \to X$ and 
$\jmath \colon X \cap Z_\reg  \to Z_\reg$
the inclusions.
Then Theorems~\ref{thm:nondegls2smooth}
and~\ref{thm:nondeg2trop}~(ii)
yield a well defined pull back homomorphism
$$
\WDiv^{\TT}(Z) = \WDiv^{\TT}(Z_\reg) \ \to \ \WDiv(X),
\qquad
D \ \mapsto \ D \vert_X  \:= \ \imath_* \jmath^*D, 
$$
where we set $\TT = \TT^n$ for short.
By Theorem~\ref{thm:nondeg2trop}~(i),
this pull back sends any invariant prime divisor
on $Z$ to a sum of distinct prime divisors
on $X$.
Moreover, we obtain a well defined induced pullback
homomorphism for divisor classes
$$
\Cl(Z) \ \to \ \Cl(X),
\qquad
[D] \ \mapsto \ [D] \vert_X .
$$
\end{remark}

The remaining ingredients are the
adjunction formula given in
Proposition~\ref{prop:adjunction}
and Proposition~\ref{prop:can-restr}
providing canonical divisors which
are suitable for the ramification
formula.

\begin{proposition}
\label{prop:adjunction}
Let $X \subseteq Z$ be an irreducible
non-degenerate toric complete intersection
given by a system $F = (f_1,\ldots,f_s)$ in
$\LP(n)$.
\begin{enumerate}
\item
Let $C_j \in \WDiv(Z)$ be the push
forward of $\mathrm{div}(f_j)$
and $K_{Z}$ an invariant canonical
divisor on $Z$.
Then the canonical class of $X$ is given
by 
$$
[K_X]
\ = \
[K_{Z}  +  C_1 + \ldots + C_s] \vert_X
\ \in \
\Cl(X).
$$
\item
The variety $X$ is $\QQ$-Gorenstein
if and only if $Z_X$ is $\QQ$-Gorenstein.
If one of these statements holds,
then $X$ and $Z_X$ have the same
Gorenstein index.
\end{enumerate}
\end{proposition}

\begin{proof}
Due to Theorem~\ref{thm:nondegls2smooth}
and Theorem~\ref{thm:nondeg2trop}~(ii)
it suffices to have the desired canonical
divisor on $Z_\reg \cap X \subseteq X_\reg$.
By Theorem~\ref{thm:nonde2ci}, the classical
adjunction formula applies, proving~(i).
For~(ii), note that the divisors
$C_j$ on $Z$ are base point free by
Lemma~\ref{lem:descrPolyOfHypersurfaces}~(v)
and hence Cartier. The assertions of~(ii)
follow.
\end{proof}

\begin{proposition}
\label{prop:can-restr}
Consider an irreducible non-degenerate
system $F$ in $\LP(n)$, a refinement 
$\Sigma' \to \Sigma$ of $F$-fans
and the associated modifications 
$\pi \colon Z' \to Z$ and $\pi \colon X' \to X$. 
Then, for every $\sigma \in \Sigma_X$, 
there are canonical divisors
$K_X(\sigma)$ on $X$ and $K_{X'}(\sigma)$ 
on $X'$ such that
\begin{enumerate}
\item
$K_{X'}(\sigma) = \pi^*K_X(\sigma)$ holds
on $X' \setminus Y'$, where $Y' \subseteq Z'$
is the exceptional locus of the toric
modification $\pi \colon Z' \to Z$,
\item
$K_{X'}(\sigma)-\pi^*K_X(\sigma)
= 
K_{Z'} \vert_{X'} - \pi^*K_Z \vert_{X'}$
holds on $\pi^{-1}(Z_\sigma) \cap X'$,
where $Z_\sigma \subseteq Z_X$ is the affine
toric chart defined by $\sigma \in \Sigma_X$.
\end{enumerate}
\end{proposition}

\begin{proof}
Fix $\sigma \in \Sigma_X$. 
Then there is a vertex $u \in B$
of the Newton polytope $B = B(F)$ 
such that the maximal cone 
$\sigma(u) \in \Sigma(B)$
contains $\sigma$.
Write $u = u_1 + \ldots + u_s$ 
with vertices $u_j \in B(f_j)$. 
With the corresponding vertices 
$a(u_j) = P^*{u_j} + a_j$ of 
the Newton polytopes $B(g_j)$,
we define 
$$ 
D(\sigma,j) 
\ := \
a(u_j)_1D_1 + \ldots + a(u_j)_rD_r
\ \in \ 
\WDiv(Z).
$$
Let $C_j \in \WDiv(Z)$ be the push forward
of $\ddiv(f_j)$.
Propositions~\ref{prop:basepointfree}
and~\ref{prop:normfanbpfample} together with
Lemma~\ref{lem:descrPolyOfHypersurfaces}~(v) 
tell us
$$
[D(\sigma,j)] 
 = 
[C_j]
 =  
\deg(g_j)
 \in 
K
 =  
\Cl(Z),
\qquad
\supp(D(\sigma,j)) \cap Z_{\sigma} 
 = 
\emptyset.
$$  
Also for the $\Sigma'$-homogenization 
$G'$ of $F$, the vertices $u_j \in B(f_j)$ 
yield corresponding vertices $a'(u_j) \in B(g_j')$ 
and define divisors
$$
D'(\sigma,j) 
\ := \ 
a'(u_j)_1D_1 + \ldots +  a'(u_j)_{r+l}D_{r+l}
\ \in \ 
\WDiv(Z').
$$
As above we have the push forwards 
$C'_j \in \WDiv(Z')$ of $\ddiv(f_j)$
and, by the same arguments, we obtain 
$$
[D'(\sigma,j)] 
 =  
\deg(g_j')
 \in 
K'
 =  
\Cl(Z'),
\qquad
\supp(D'(\sigma,j)) \cap \pi^{-1}(Z_{\sigma}) 
 =  
\emptyset.
$$
Take the invariant canonical divisors $K_Z$ on $Z$ 
and $K_{Z'}$ in $Z'$ with multiplicity~$-1$ along
all invariant prime divisors and set
$$
K_{X}(\sigma) \ := \ (K_{Z} + \sum_{j=1}^s D(\sigma,j)) \vert_{X},
\qquad\quad
K_{X'}(\sigma) \ := \ (K_{Z'} + \sum_{j=1}^s D'(\sigma,j)) \vert_{X'}.
$$
According to Proposition~\ref{prop:adjunction},
these are canonical divisors on $X$ and $X'$ 
respectively.
Properties~(i) and~(ii) are then clear
by construction.
\end{proof}

\begin{proof}[Proof of Theorem~\ref{thm:main-1}]
First observe that $\mathcal{A}_{X}$ is an
anticanonical complex for the toric variety $Z_X$.
Now, choose any regular refinement 
$\Sigma' \to \Sigma$ of the defining 
$F$-fan~$\Sigma$ of the irreducible 
non-degenerate toric complete intersection
$X \subseteq Z$.
This gives us modifications $\pi \colon Z' \to Z$ 
and $\pi \colon X' \to X$.
Standard toric geometry and 
Theorem~\ref{thm:nondegls2smooth}
yield that both are resolutions of 
singularities. 

Proposition~\ref{prop:can-restr}
provides us with canonical divisors
on $X'$ and $X$.
We use them to compute discrepancies.
Over each $X \cap Z_\sigma$, where
$\sigma \in \Sigma_X$, we obtain
the discrepancy divisor as
$$
K_{X'}(\sigma)-\pi^*K_X(\sigma)
\ = \ 
K_{Z'}\vert_X - \pi^*K_{Z_X}\vert_X.
$$
By Theorem~\ref{thm:nondeg2trop}~(i), every
exceptional prime divisor $E_X' \subseteq X'$
admits a unique exceptional prime divisor
$E_Z' \subseteq Z'$ with $E_X' \subseteq E_Z'$.
Remark~\ref{rem:tci-weilpull} 
guarantees that the
discrepancy of $E_X'$ with
respect to $\pi \colon X' \to X$
and that of $E_Z'$ with respect to 
$\pi \colon Z' \to Z_X$
coincide.
\end{proof}

We conclude the section by discussing the divisor
class group and the Cox ring of a non-degenerate 
complete toric intersection and the effect of
a general choice of the defining Laurent system.

\begin{proposition}
\label{prop:coxringcrit}
Consider a non-degenerate system $F = (f_1,\ldots,f_s)$ 
in $\LP(n)$, an $F$-fan $\Sigma$ in $\ZZ^n$ and
the associated toric complete intersection $X \subseteq Z$.
Assume that $\bar X \setminus \hat X$ is of codimension
at least two in $\bar X$. 
If the pullback $\Cl(Z) \rightarrow \Cl(X)$ is an isomorphism,
then the Cox ring of $X$ is given by
$$
\mathcal{R}(X)
\ = \
\KK[T_1, \ldots, T_r] / \bangle{g_1, \ldots, g_s},
\qquad
\deg(T_i) \ = \ [X_i] \ \in \ \Cl(X),
$$
where $G = (g_1,\ldots,g_s)$ is the $\Sigma$-homogenization
of $F$.
In this situation, we have moreover the following statements.
\begin{enumerate}
\item 
If $\bar X \cap V(T_i) \setminus \hat X$ is of codimension at least
two in $\bar X \cap V(T_i)$, then $T_i$ defines a prime element in
$\mathcal{R}(X)$.
\item
If $\deg(g_j) \ne \deg(T_i)$ holds for all $i,j$, then
the variables $T_1, \ldots, T_r$ define a minimal system
of generators for $\mathcal{R}(X)$.
\end{enumerate}
\end{proposition}

\begin{proof}
According to Theorem~\ref{thm:nonde2ci}~(iv) ensures
that $\bar X$ is normal.
This allows us to apply~\cite[Cor.~4.1.1.5]{ArDeHaLa},
which shows that the Cox ring $\mathcal{R}(X)$ is
as claimed.
The supplementary assertion~(i) is a
consequence of Proposition~\ref{prop:piecewiseci}
and~(ii) is clear.
\end{proof}

\begin{definition}
\label{def:general}
Let $B_1, \ldots, B_s \subseteq \QQ^n$ be 
integral polytopes.
The \emph{Laurent space} associated with
$B_1, \ldots, B_s$ is the finite-dimensional
vector space
$$
V(B_1, \ldots, B_s)
\ := \
\bigoplus_{j=1}^s \KK[T^{\nu}; \, \nu \in B_j \cap \ZZ^n].
$$
Given a non-empty open set $U \subseteq V(B_1, \ldots, B_s)$,
we refer to the elements $F \in U$ and also to the
possible associated toric complete intersections as
\emph{$U$-general}.
\end{definition}

Following common (ab)use, we say ``the general Laurent system
$F = (f_1, \ldots,f_s)$ in $\LP(n)$ satisfies ...''
if we mean  ``there is a $U \subseteq V(B_1,\ldots,B_s)$
such that every $U$-general $F$ satisfies ...'', where
$B_j$ denotes the Newton polytope of $f_j$ for $j = 1, \ldots, s$.
By~\cite[Thm.~2.2]{Kh}, the general Laurent system
is non-degenerate.

\begin{corollary}
\label{cor:general2SrRa}
Let $F = (f_1,\ldots,f_s)$ be a general Laurent system
in $\LP(n)$ and~$\Sigma$ an $F$-fan in $\ZZ^n$.
For the associated toric complete intersection 
$X \subseteq Z$ assume
$$
\dim(\bar Z \setminus \hat Z)
\ \le \
r-s-2,
\qquad\qquad
\dim(X) \ \ge \ 3.
$$
Then  the variety $X$ is irreducible and normal, 
the pullback $\Cl(Z) \to \Cl(X)$ is an isomorphism
and the Cox ring of $X$ is given as
$$
\mathcal{R}(X)
 = 
\KK[T_1, \ldots, T_r] / \bangle{g_1, \ldots, g_s},
\qquad
\deg(T_i)  =  [D_i] \ \in \ \Cl(X)  = \Cl(Z),
$$
where $G = (g_1,\ldots,g_s)$ is the $\Sigma$-homogenization
of $F = (f_1,\ldots,f_s)$ and $D_i \subseteq Z$ the toric prime divisor
corresponding to 
$T_i \in \mathcal{R}(Z) = \KK[T_1, \ldots, T_r]$.
\end{corollary}

The proof of this Corollary is
covered by the subsequent two remarks, which we
formulate separately as they touch
aspects of independent interest.

\begin{remark}
Let $F = (f_1,\ldots,f_s)$ be a Laurent system
in $\LP(n)$ and~$\Sigma$ an $F$-fan in $\ZZ^n$.
Then Lemma~\ref{lem:descrPolyOfHypersurfaces}~(ii)
tells us that $F$ is general if and only if its
$\Sigma$-homogenization $G$ is general. 
\end{remark}

The second remark shows in particular
that the easy-to-check assumption
$\dim(\bar Z \setminus \hat Z) \le r-s-2$
might even be weakened.

\begin{remark}
\label{rem:general2SrRa}
Consider a toric variety~$Z$ and a non-degenerate toric 
complete intersection $X = X_1 \cap \ldots \cap X_s$ 
in~$Z$ of dimension at least three.
Then $X$ is constructed by passing stepwise to 
hypersurfaces: 
$$
X_0' \ := \ Z,
\qquad  
X_{j}' \ := \ X_{j-1}' \cap X_{j} \ \subseteq \ Z,
\qquad 
j = 1, \ldots, s.
$$ 
Then $X = X_s'$ and each $X_j'$ is a non-degenerate
toric complete intersection in $Z$. 
In each step, $X_j \vert_{X_{j-1}'}$ is a base point free 
divisor on $X_{j-1}'$; see~Lemma~\ref{lem:descrPolyOfHypersurfaces}.
The Grothendieck-Lefschetz Theorem 
from~\cite{RaSr} provides us with  
a pullback isomorphism 
$$ 
\Cl(X_{j-1}') \ \to \ \Cl(X_j')
$$
for a general choice of $X_j \vert_{X_{j-1}'}$
with respect to its linear system.
In the initial step, the linear system of $X_1$ 
is just the projective space over the corresponding 
homogeneous component of the Cox ring, that 
means that we have
$$ 
\vert X_1 \vert 
\ = \ 
\PP \, \mathcal{R}(Z)_{\left[ X_1 \right]}.  
$$
Now consider a general  
$X = X_1 \cap \ldots \cap X_s \subseteq Z$
and suppose that $\bar X_j' \setminus \hat X_j'$ 
is of codimension at least two in $\bar X_j'$
in each step. 
Then we may apply Proposition~\ref{prop:coxringcrit} 
stepwise, where in each step we observe 
$$ 
\mathcal{R}(X_{j-1}')_{\bigl[X_j \vert_{X_{j-1}'}\bigr]}  
\ = \ 
\mathcal{R}(Z)_{\left[X_j\right]}
/
\bangle{g_1, \ldots, g_{j-1}}_{\left[X_j\right]} .
$$
Thus the general choice of 
$X_1 \cap \ldots \cap X_s \subseteq Z$ induces
a general choice of the divisor $X_j \vert_{X_{j-1}'}$
on $X_{j-1}'$ in each step.  
In particular, we obtain $\Cl(Z) = \Cl(X)$ and see
that the statements of 
Proposition~\ref{prop:coxringcrit} apply to
$X \subseteq Z$.
\end{remark}

\begin{example}
\label{ex:gentype}
Corollary~\ref{cor:general2SrRa} enables us
to produce Mori dream spaces with prescribed
properties. For instance, consider general
toric hypersurfaces
$$
X
\ = \
V(f)
\ \subseteq \
\PP_{1,1,2} \times \PP_{1,1,2}
\ = \
Z,
$$
where $f$ is $\ZZ^2$-homogeneous of bidegree $(d_1,d_2)$ with
$d_1,d_2 \in \ZZ_{\ge 1}$.
Corollary~\ref{cor:general2SrRa} directly
yields $\Cl(X) = \ZZ^2$ and delivers the Cox ring as
$$
\mathcal{R}(X)
\ = \
\KK[T_0,T_1,T_2,S_0,S_1,S_2] / \bangle{f},
\qquad
\begin{array}{rr}
w_0=w_1=(1,0), & w_2=(2,0),
\\
u_0=u_1=(0,1), & u_2=(0,2),
\end{array}                   
$$
where $w_i = \deg(T_i)$ and $u_i = \deg(S_i)$.
Corollary~\ref{cor:acc-singu} tells
us that $X$ has worst canonical singularities.
Moreover, if for instance $d_1 = d_2 = d$,
then in the cases
$$
d > 4, \qquad
d = 4, \qquad
d < 4,
$$
the Mori dream space $X$ is of general type, 
satisfies $\mathcal{K}_X = 0$ or is Fano,
accordingly; use Proposition~\ref{prop:adjunction}.
\end{example}

\section{Fake weighted terminal Fano threefolds}

Here we prove Theorem~\ref{thm:terminalWHf-new}.
The first and major part uses the whole theory
developed so far to establish suitable upper
bounds on the specifying data.
Having reduced the problem to working out a
managable number of cases, we proceed
computationally, which involves besides a
huge number of divisibility checks the
search for lattice points inside polytopes
tracing back to the terminality criterion
provided in Corollary~\ref{cor:acc-singu}.
A second and minor part concerns the verifying and
distinguishing items listed in
Theorem~\ref{thm:terminalWHf-new},
where we succeed with Corollary~\ref{cor:general2SrRa}
and the computation of suitable invariants.

We fix the notation around a non-degenerate complete
intersection $X$ in an $n$-dimensional fake weighted
projective space $Z$. 
The defining fan of $\Sigma$ in $\ZZ^n$ is simplical, 
complete and we denote its primitive generators 
by $v_0,\ldots, v_n$. 
The divisor class group $\Cl(Z)$ is of the form
$$
\Cl(Z) 
\ = \ 
\ZZ \times \ZZ / t_1 \ZZ \times \ldots \times \ZZ / t_q \ZZ.
$$
By $w_i = (x_i, \eta_{i1}, \ldots, \eta_{iq}) \in \Cl(Z)$
we denote the classes of the torus invariant prime divisors
$D_i$ on $Z$.
Recall that any $n$ of $w_0, \ldots, w_n$ generate
$\Cl(Z)$.
Moreover, as the $\Cl(Z)$-grading on $\mathcal{R}(Z)$ is
pointed, we may assume
$$
0 
\ < \ 
x_0 
\ \le \ \ldots \ \le \ 
x_n. 
$$
As before, $X \subseteq Z$ arises from a Laurent system~$F$ 
in $\LP(n)$ and $\Sigma$ is an $F$-fan.
We denote by $G = (g_1,\ldots,g_s)$ the $\Sigma$-homogenization
of $F = (f_1,\ldots, f_s)$.
Recall that the $\Cl(Z)$-degree
$\mu_j = (u_j, \zeta_{j1}, \ldots, \zeta_{jq})$
of $g_j$ is base point free.

\begin{lemma}
\label{lem:fwpsfromlaurent}
A divisor class $[D] \in \Cl(Z)$ is base point free if 
and only if for any $i = 0, \ldots, n$ there exists 
an $l_i \in \ZZ_{\ge 1}$ with $[D] = l_i w_i \in \Cl(Z)$.
\end{lemma}

\begin{proof}
This is a direct consequence of
Proposition~\ref{prop:basepointfree}
and the fact that the maximal cones of $\Sigma$ are given by
$\cone(v_j; \ j \ne i)$ for $i = 0, \ldots, n$.
\end{proof}

The following lemma provides effective
bounds on the orders $t_1, \dotsc, t_q$ 
of the finite cyclic components of $\Cl(Z)$
in terms of the $\ZZ$-parts $x_i$ of the generator 
degrees $w_0, \ldots, w_n$ and $u_j$ of 
the relation degrees $\mu_1, \ldots, \mu_s$
for any toric complete intersection $X$ 
in a weighted projective space~$Z$ with 
$x_0 =1$.

\begin{lemma}
\label{lem:torsbound}
Assume $x_0 = 1$. Moreover, let 
$\mu = (u, \zeta_1, \ldots, \zeta_q) \in \Cl(Z)$
be a base point free divisor class. 
Then, for any $k = 1, \ldots, q$ and 
$j = 1, \ldots, n$, we have
$$
t_k 
\ \mid \ 
\lcm \left(\frac{u}{x_i}; \ i = 1, \ldots, n, \ i \ne j \right).
$$
In particular all $t_k$ divide $u$. Moreover, for the 
$\ZZ$-parts~$u_j$ of the relation degrees~$\mu_j$, 
we see that each of 
$t_1, \ldots, t_q$ divides $\gcd(u_1, \ldots, u_s)$.
\end{lemma}

\begin{proof}
Due to $x_0 = 1$, we may assume
$\eta_{11} = \ldots = \eta_{1q} = 0$.
Lemma~\ref{lem:fwpsfromlaurent} delivers
$l_i \in \ZZ_{\ge 1}$ with $\mu = l_iw_i$.
For $i = 0, \ldots, n$ that means  
$$
(l_0,0,\ldots,0)
\ = \
l_1w_1
\ = \
\mu
\ = \
l_iw_i
\ = \
(l_ix_i,l_i\eta_{i1},\ldots,l_i\eta_{iq}).
$$
Thus, we always have $u = l_ix_i$ and
$l_i \eta_{ik} = 0$.
Now, fix $1 \le j \le n$.
As any $n$ of the $w_i$ generate $\Cl(Z)$,
we find $\alpha \in \ZZ^{n+1}$ with
$\alpha_j = 0$ and 
$$
\alpha_0 w_0 + \ldots + \alpha_n w_n
\ = \
(1, \bar 1, \ldots, \bar 1).
$$
Scalar multiplication of both sides with
$\lcm(l_i; \ 1 \le i \le n, \ i \ne j)$
gives the first claim.
The second one is clear.
\end{proof}

The next bounding lemma uses terminality. 
Given $\sigma \in \Sigma$, let $I(\sigma)$ 
be the set of indices such that the~$v_i$ 
with $i \in I(\sigma)$ are precisely
the primitive ray generators of~$\sigma$
and $u_\sigma \in \QQ^n$ a linear form
evaluating to~$-1$ on each $v_i$ with 
$i \in I(\sigma)$. As before, we look at    
$$ 
A(\sigma) 
\ := \ 
\{v \in \sigma; \ 0 \ge \bangle{u_\sigma,v} \ge -1\}
\ \subseteq \
\sigma.
$$
The point $z_\sigma \in Z$ is at most a terminal 
singularity of $Z$ if and only if $0$ and the 
$v_i$ with $i \in I(\sigma)$ are the only lattice 
points in~$A(\sigma)$.  
According to Theorem~\ref{thm:main-1}, the analogous
statement holds for the points $x \in X$ with 
$x \in \TT^n \cdot z_\sigma$.

\goodbreak

\begin{lemma}
\label{lem:terminalprops}
Consider $\sigma \in \Sigma$ such that $z_\sigma \in Z$ 
is at most a terminal singularity of $Z$.
\begin{enumerate}
\item
If $\sigma$ is of dimension two, then $\sigma$ 
is a regular cone and $\Cl(Z)$ is generated 
by the $w_i$ with $i \not\in I(\sigma)$.
In particular, $\gcd(x_i; \ i \not\in I(\sigma)) = 1$ 
holds.
\item
If $\sigma$ is of dimension at least two, then
$\gcd(x_i; \ i \not\in I(\sigma))$ is strictly less 
than the sum over all $x_i$ with $i \in I(\sigma)$.
\end{enumerate}
\end{lemma}

\begin{proof}
The first assertion can easily be verified directly.
We turn to the second one.
Using $x_i \in \ZZ_{\ge 1}$ and $x_0 v_0 + \ldots + x_n v_n = 0$,
we obtain
$$
v' 
\ \coloneqq \ 
- \sum_{i \notin I(\sigma)} x_i v_i 
\ = \ 
\sum_{i \in I(\sigma)} x_i v_i 
\ \in \ 
\sigma^\circ \cap \ZZ^n.
$$
Write $v' = \gcd(x_i; \ i \not\in I(\sigma)) v$ with 
$v \in \sigma^\circ \cap \ZZ^n$.
Due to $\dim(\sigma) \ge 2$, the vector $v$ does 
not occur among $v_0, \ldots, v_n$. 
Evaluating $u_\sigma$ yields
$$ 
0 
\ \ge \ 
\bangle{u_\sigma,v}
\ = \ 
\gcd(x_i; \ i \not\in I(\sigma))^{-1}\bangle{u_\sigma,v'}
\ = \ 
- \gcd(x_i; \ i \not\in I(\sigma))^{-1}\sum_{i \in I(\sigma)} x_i .
$$
By assumption, we have $v \not\in A(\sigma)$.
Consequently, the right hand side term is 
strictly less than~$-1$. This gives us the desired 
estimate.
\end{proof}

We turn to bounds involving the Fano property
of a toric complete intersection threefold $X$ in
a fake weighted projective space $Z$.
A tuple $\xi = (x_0, \ldots, x_n)$ of positive integers
is ordered if $x_0 \le \ldots \le x_n$ holds and
well-formed if any $n$ of its entries are coprime.
For an ordered tuple $\xi$, we define
$$
m(\xi)
\ := \
\lcm(x_0, \ldots, x_n),
\qquad\qquad
M(\xi)
\ := \
\begin{cases}
2 m(\xi), & x_{n} = m(\xi),
\\
m(\xi), & x_n \ne m(\xi).
\end{cases}
$$
We deal with well-formed ordered tuples
$\xi = (x_0, \ldots, x_n)$ with $n \ge 4$.
As we will see, the Fano property forces
the inequality
\begin{equation}
\label{eq:abstractFano}
(n-3)M(\xi)
\ < \
x_0 + \ldots + x_n .
\end{equation}

\begin{lemma}
\label{lem:wf5-tuples}
Consider an ordered $\xi = (x_0, \ldots, x_4)$
such that any three of $x_0, \ldots, x_4$ are
coprime and condition~(\ref{eq:abstractFano})
is satisfied. 
Then $x_4 \le 41$ holds or we have
$1 \le x_0, x_1, x_2 \le 2$ and $x_3 = x_4$.
\end{lemma}

\begin{proof}
We first settle the case $x_4 = m(\xi)$.
Then $x_4$ is divided by each of
$x_0, \ldots, x_3$.
This implies
$$
\gcd(x_i, x_j)
\ = \
\gcd(x_i, x_j, x_4)
\ = \ 1,
\qquad
0 \le i < j \le 3.
$$
Consequently, $x_0 \cdots x_3$ divides $x_4$.
Subtracting $x_4$ from both sides of
the inequality~(\ref{eq:abstractFano})
leads to
$$
 x_0 \cdots x_3
\ \le \
x_4
\ < \
x_0 + \ldots + x_3.
$$
Using $1 \le x_0 \le \ldots \le x_3$ and
pairwise coprimeness of $x_0, \ldots, x_3$, we
conclude that the tuple $(x_0,x_1,x_2,x_3)$
is one of 
$$
(1,1,2,3),
\qquad\qquad
(1,1,1,x_3).
$$
In the first case, we arrive at
$x_4 < x_0 + \ldots + x_3 = 7$.
In the second one, $x_4 = dx_3$
holds with $d \in \ZZ_{\ge 1}$.
Observe
$$
dx_3
\ = \ 
x_4
\ < \
x_0 + \ldots + x_3
\ = \
3 + x_3.
$$ 
Thus, we have to deal with $d = 1,2,3$.
For $d=1$, we arrive at $x_3=x_4$
and the cases $d=2,3$ lead to $x_3 \le 2$
which means $x_4 < 5$.

Now we consider the case $x_4 < m(\xi)$.
Then $m(\xi) = lx_4$ with $l \in \ZZ_{\ge 2}$.
From inequality~(\ref{eq:abstractFano}) we
infer $l \le 4$ as follows:
$$
lx_4
\ = \ 
m(\xi)
\ < \
x_0 + \ldots + x_4
\ \le \ 5x_4.
$$
We first treat the case $x_3 = x_4$.
Using the assumption that any three
of $x_0, \ldots, x_4$ are coprime,
we obtain
$$
\gcd(x_i,x_4)
\ = \
\gcd(x_i,x_3,x_4)
\ = \ 1,
\qquad
i = 0,1,2.
$$
Consequently, $x_2x_4 \le m(\xi) = lx_4$
and $x_2 \le l \le 4$.
For $l = 2$ this means $1 \le x_0,x_1,x_2 \le 2$.
For $l = 3,4$, we use again~(\ref{eq:abstractFano})
and obtain
$$
x_4
\ < \
\frac{1}{l-2} (x_0+x_1+x_2)
\ \le \
12.
$$
Now we turn to the case $x_3 < x_4$. 
Set for short $d_i \coloneqq \gcd(x_i, x_4)$.
Then, for all $0 \le i < j \le 3$, we observe
$$
\gcd(d_i, d_j)
\ = \
\gcd(x_i, x_j, x_4)
\ = \
1.
$$
Consequently $d_{0} \cdots d_{3} \mid x_4$.
For $i = 0, \ldots, 3$, write $x_i = f_i d_i$ 
with $f_i \in \ZZ_{\ge 1}$. Then~$f_i$ divides 
$lx_4$ and hence $l$.
Fix $i_0, \ldots, i_3$ pairwise distinct with 
$d_{i_0} \le \ldots \le d_{i_3}$. 
Using~(\ref{eq:abstractFano}), we obtain
$$
(l-1) d_{i_0} \cdots d_{i_3} 
\ \leq \ 
(l-1) x_4
\ < \ 
f_{i_0} d_{i_0} + \ldots + f_{i_3} d_{i_3}
\ \le  \ 
(2 + 2l) d_{i_3}.
$$
For the last estimate, observe that due to $l = 2,3,4$, 
all $f_i \ne 1$ have a common factor 2 or 3. 
Thus, as any three of $x_0, \ldots, x_3$ are coprime,
we have $f_i = 1$ for at least two~$i$.
We further conclude
$$
d_{i_0} d_{i_1} d_{i_2} 
\ < \ 
\frac{(2+ 2l)}{l -1} 
\ \le \
6.
$$
This implies $d_{i_0} = d_{i_1} = 1$ and $d_{i_2} \le 5$.
We discuss the case $f_{i_3} = 1$. There, we have
$x_{i_3} = d_{i_3}$, hence $x_{i_3} \mid x_4$.
By assumption, $x_0 \le \ldots \le x_3 < x_4$ and
thus $x_{i_3} < x_4$. We conclude $d_{i_3} = x_{i_3} \le x_4 / 2$.
From above we infer
$$
(l-1) x_4
\ < \ 
f_{i_0} d_{i_0} + \ldots + f_{i_3} d_{i_3}
\ \le \ 
l(2 + d_{i_2}) + \frac{x_4}{2}.
$$
Together with $l = 2,3,4$ and $d_{i_2} \le 5$ as observed 
before, this enables us to estimate~$x_4$ as follows:
$$
x_4
\ < \ 
2l \frac{2+ d_{i_3}}{2l - 3}
\ \le \ 
28.
$$
Now let $f_{i_3} > 1$. 
Then
$2 d_{i_3} \le f_{i_3} d_{i_3} = x_{i_3} < x_4$
holds.
This gives $d_{i_3} < x_4 / 2$.
Using $d_{i_3} \mid x_4$ we conclude 
$d_{i_3} \le x_4 / 3$.
Similarly as before, we proceed by
$$	
(l-1) x_4
\ < \
f_{i_0} d_{i_0} + \dotsb + f_{i_3} d_{i_3}
\ \le \  
l(2 + d_{i_2}) + l d_{i_3}
\ \leq \ 
l(2 + d_{i_2}) + l\frac{x_4}{3}.
$$
Again, inserting $l = 2,3,4$ and the bound 
$d_{i_2} \le 5$ finally leads to the desired 
estimate
$$
x_4 
\ < \ 
3l\frac{2 + d_{i_2}}{2l -3}
\ \le \
42.
$$
\end{proof}

\begin{lemma}
\label{lem:wf6tuples-1}
Consider a well-formed ordered $\xi = (x_0, \dotsc, x_5)$
satisfying~(\ref{eq:abstractFano}).
Then $x_5 \le 21$ holds or we have $1 \le x_0, x_1 \le 2$ 
and $x_2 = x_3 = x_4 = x_5$.
\end{lemma}

\begin{proof}
Let $x_5 \ge 22$.
We have $M(\xi) = lx_5$ with $l \ge 2$.
From~(\ref{eq:abstractFano}) we infer
$2 l x_5 < 6x_5$, hence $l=2$.
Thus, we can reformulate~(\ref{eq:abstractFano})
as
$$
3x_5
\ < \
x_0 + \ldots + x_4.
$$
Moreover, $M(\xi) = 2x_5$ implies $a_ix_i = 2x_5$
with suitable $a_i \in \ZZ_{\ge 2}$ for
$i = 0,\ldots, 4$.
In particular, the possible values of $x_0,\ldots,x_4$
are given as
$$
x_5,
\quad
\frac{2}{3}x_5,
\quad
\frac{1}{2}x_5,
\quad
\frac{2}{5}x_5,
\quad
\frac{1}{3}x_5,
\quad
\frac{2}{7}x_5,
\quad
\ldots \ .
$$
We show $x_4 = x_5$. Suppose $x_4 < x_5$.
Then $x_4 \le 2x_5/ 3$.
We have $x_1 \ge 2x_5/ 3$, because otherwise
$x_1 \le x_5/2$ and thus
$$
3x_5
\ < \
x_0 + \ldots + x_4
\ \le \ 
\frac{1}{2}x_5 + \frac{1}{2}x_5 +
\frac{2}{3}x_5 + \frac{2}{3}x_5 + \frac{2}{3}x_5
\ = \
3 x_5,
$$
a contradiction. We conclude
$x_1 = \ldots = x_4 = 2x_5/ 3$.
By well-formedness, the integers
$x_1,\ldots,x_5$ are coprime.
Combining this with 
$$
3x_1
\ = \ \ldots \ = \
3x_4
\ = \
2x_5
$$
yields $x_5 = 3$, contradicting $x_5 \ge 22$.
Thus, $x_4 = x_5$, and we can update
the previous reformulation
of~(\ref{eq:abstractFano})~as
$$
2x_5
\ < \
x_0 + \ldots + x_3.
$$
We show $x_3 = x_5$. Suppose $x_3 < x_5$.
Then, by the limited stock of possible values
for the $x_i$, the displayed inequality
forces $x_3 = 2x_5/3$ and one of the following
$$
x_2 \ = \ \frac{2}{3}x_5,
\quad
x_1 \ = \ \frac{2}{3}x_5,  \frac{1}{2}x_5, \ \frac{2}{5}x_5,
\qquad \qquad
x_2 \ = \  \frac{1}{2}x_5,
\quad
x_1 \ = \ \frac{1}{2}x_5.
$$
By well-formedness, $x_1, \ldots, x_5$ are coprime.
Depending on the constellation,
this leads to $x_5 = 3, 6, 15$,
contradicting $x_5 \ge 22$. 
Thus, $x_3=x_5$.
Observe
$$
x_5
\ < \
x_0 + x_1 + x_2,
\qquad
\gcd(x_i,x_j,x_5) = 1,
\quad
0 \le i < j \le 2.
$$
We show $x_2 = x_5$ by excluding all values $x_2 < x_5$.
First note $x_2 > x_5/3$.
Assume $x_2 = 2x_5/5$.
Then, by the above inequality, $x_1 = 2x_5/5$.
We obtain
$$
5x_1 \ = \ 5x_2 \ = \ 2x_5 .
$$
thus $\gcd(x_1,x_2,x_5) = 1$ implies $x_5 = 5$,
a contradiction to $x_5 \ge 22$.
Next assume $x_2 = x_5/2$.
The inequality leaves us with
$$
x_1
\ = \
\frac{1}{2}x_5, \
\frac{2}{5}x_5, \
\frac{1}{3}x_5, \
\frac{2}{7}x_5.
$$
Thus, using $\gcd(x_1,x_2,x_5) = 1$ we arrive
at $x_5 = 2,10,6,14$ respectively, contradicting
$x_5 \ge 22$. Finally, let $x_2 = 2x_5/3$.
Then we have to deal with 
$$
x_1
\ = \
\frac{2}{3}x_5, \
\frac{1}{2}x_5, \
\frac{2}{5}x_5, \
\frac{1}{3}x_5, \
\frac{2}{7}x_5, \
\frac{1}{4}x_5, \
\frac{2}{9}x_5, \
\frac{1}{5}x_5, \
\frac{2}{11}x_5.
$$
Using $\gcd(x_1,x_2,x_5) = 1$ gives
$x_5 = 3,6,15,3,21,12,9,15$
in the first eight cases,
excluding those.
Thus, we are left with
the three cases
$$
x_2 \ = \ \frac{2}{3}x_5,
\qquad
x_1 \ = \ \frac{2}{11}x_5,
\qquad
x_0 \ = \ \frac{2}{11}x_5, \ \frac{1}{6}x_5, \ \frac{2}{13}x_5.
$$
In the first one, coprimeness of $x_0,x_1,x_5$ gives
$x_5 = 11$ and in the second one
coprimeness of $x_0,x_2,x_5$ implies $x_5 = 6$.
The third case is excluded by
$$
\gcd(x_1,x_2,x_5) = 1 \ \Rightarrow \ x_5 = 33,
\qquad \qquad
\gcd(x_0,x_2,x_5) = 1 \ \Rightarrow \ x_5 = 39.
$$
Thus, $x_2 = x_5$.
We care about $x_0$ and $x_1$.
Well-formedness and $x_2 = \ldots = x_5$
yield that $x_0,x_5$ as well as $x_1,x_5$
are coprime.
Thus, we infer $1 \le x_0,x_1 \le 2$ from
$$
a_0 x_0 \ = \ 2 x_5,
\qquad \qquad
a_1 x_1 \ = \ 2 x_5.
$$
\end{proof}

\begin{lemma}
\label{lem:wf7tuples}
There exist only two ordered well-formed septuples
$(x_0, \ldots, x_6)$ satisfying~(\ref{eq:abstractFano}),
namely
$(1, 1, 1, 1, 1, 1, 1)$ and
$(2, 2, 3, 3, 3, 3, 3)$.
\end{lemma}

\begin{proof}
The case $x_6 = 1$ gives the first tuple.
Let $x_6 > 1$. Then $M(\xi) = l x_6$ holds
with $l \ge 2$.
Using~(\ref{eq:abstractFano}), we see
$3lx_6 < 7x_6$ which means $l = 2$.
We obtain
$$
5 x_6 \ < \ x_0+ \ldots + x_5
$$
by adapting the inequality~(\ref{eq:abstractFano})
to the present setting.
Similar to the preceding proof, $M(\xi) = 2x_6$
leads to presentations
$$
x_i
\ = \
\frac{2}{a_i} x_6,
\qquad
a_i \in \ZZ_{\ge 2},
\quad
i = 0,\ldots, 5.
$$
Now, pick the unique $j$ with
$x_0 \le \ldots \le x_{j-1} < x_j = \ldots = x_6$.
Well-formedness implies $j \ge 2$.
Moreover $x_{j-1} \le 2x_6/3$ holds and thus
$$
5 x_6
\ < \
\frac{2}{3}jx_6 + (6-j)x_6
\ = \
\frac{18-j}{3} x_6 .
$$
This implies $j < 3$. Thus $j=2$, which means
$x_0 \le x_1 < x_2 = \ldots = x_6$.
Adapting the inequality~(\ref{eq:abstractFano}) 
accordingly gives 
$$
x_6 \ < \ x_0+x_1.
$$
Moreover, by well-formedness, $x_0,x_6$ as well as
$x_1,x_6$ are coprime.
Consequently, we can deduce $1 \le x_0 \le x_1 \le 2$
from
$$
a_0 x_0 \ = \ 2 x_6,
\qquad \qquad
a_1 x_1 \ = \ 2 x_6.
$$
Now, $x_6 > 1$ excludes $x_1 = 1$. Next,
$x_0 = 1$ would force $x_6 = 2 = x_1$, contradicting
the choice of $j$. Thus, we arrive at $x_0 = x_1 = 2$
and $x_2 = \ldots = x_6 = 3$.
\end{proof}

The last tool package for the proof of
Theorem~\ref{thm:terminalWHf-new} supports
the verification of candidates in the sense
that it allows us to show that each of the
specifying data in the list do indeed stem
from a toric complete intersection.

\begin{reminder}
\label{rem:divpoly}
Consider any complete toric variety $Z$ arising 
from a lattice fan~$\Sigma$ in $\ZZ^n$.
With every invariant Weil divisor
$C = a_1 D_1 + \ldots + a_r D_r$ on $Z$
one associates its \emph{divisorial polytope}
$$
B(C) 
\ = \ 
\{ u \in \QQ^n; \ \bangle{u, v_i} \ge -a_i, \ i = 1, \ldots, r \}
\ \subseteq \
\QQ^n.
$$
If $C$ is base point free, then $B(C)$ has integral 
vertices and $\Sigma$ refines the normal fan of $B(C)$.
If in addition $C$ is ample, then $B(C)$ is a full-dimensional 
lattice polytope having  $\Sigma$ as its normal fan.
\end{reminder}

Given base point free classes $\mu_1, \ldots, \mu_s$
on a toric variety $Z$, the question is wether
or not these are the relation degrees of a (general)
toric complete intersection.
The following criterion relies on
Corollary~\ref{cor:general2SrRa}.

\begin{remark}
\label{rem:verfication-tools}
Consider a complete toric variety $Z$ given by
a fan~$\Sigma$ in~$\ZZ^n$ and let 
$\mu_1, \ldots, \mu_s \in \Cl(Z)$ such that 
each $\mu_j$ admits a base point free
representative $C_j \in \WDiv^{\TT}(Z)$. 
Being integral, the $B(C_j)$ can be realized 
as Newton polytopes: 
$$ 
B(C_j) \ = \ B(f_j),
\qquad
f_j \ \in \ \LP(n),
\qquad
j \ = \ 1, \ldots, s.
$$
Consider the system $F = (f_1,\ldots,f_s)$ in $\LP(n)$. 
The fan $\Sigma$ refines the 
normal fan of the Minkowski sum $B(C_1) + \ldots + B(C_s)$
and hence is an $F$-fan.
For the $\Sigma$-homogenization $G = (g_1, \ldots, g_s)$
of $F$ we have $\deg(g_j) = \mu_j \in \Cl(Z)$.
\begin{enumerate}
\item
If $F$ is non-degenerate, then the associated variety 
$X \subseteq Z$ is a toric complete intersection.
\item
If $F$ is non-degenerate, then there is a non-empty
open $U \subseteq V_F$ such that every
$G' \in U$ defines a non-degenerate $F' \in \LP(n)$.
\item 
If $F$ is general and we have
$\dim(\bar Z \setminus \hat Z) \le r-s-2$,
then $\Cl(X) = \Cl(Z)$ holds and
the Cox ring of $X$ is given as
$$
\qquad 
\mathcal{R}(X)
\  = \
\KK[T_1,\ldots,T_r] / \bangle{g_1, \ldots, g_s},
$$
where $\deg(T_i) \in \Cl(X)$ is the class $[D_i] \in \Cl(Z)$
of the invariant prime divisor $D_i \subseteq Z$ corresponding
to~$T_i$.
\end{enumerate}  
\end{remark}

\begin{proof}[Proof of Theorem~\ref{thm:terminalWHf-new}]
Let $Z$ be a fake weighted projective space arising from 
a fan~$\Sigma$ in  $\ZZ^n$ and let 
$X = X_1 \cap \ldots \cap X_s \subseteq Z$ 
be a general (non-degenerate) terminal Fano complete 
intersection threefold.
Write $G = (g_1,\ldots, g_s)$ for the $\Sigma$-homogenization
of the defining Laurent system $F = (f_1,\ldots,f_s)$ of
$X \subseteq Z$.
We have
$$ 
\Cl(Z) 
\ = \ 
\ZZ \times \ZZ / t_1 \ZZ \times \ldots \times \ZZ / t_q \ZZ 
$$
for the divisor class group of $Z$. 
As before, the generator degrees~$w_i = \deg(T_i)$ 
and the relation degrees $\mu_j = \deg(g_j)$ 
in $\Cl(Z)$ are given as  
$$ 
w_i 
\ = \ 
[D_i]
\ = \ 
(x_i, \eta_{i1}, \ldots, \eta_{iq}),
\qquad
\mu_j 
= 
[X_j] 
\ = \ 
(u_j, \zeta_{j1}, \ldots, \zeta_{jq}).
$$
We assume that the presentation
$X \subseteq Z$ is irredundant in the sense
that no $g_i$ has a monomial $T_i$; otherwise,
as the $\Cl(Z)$-grading 
is pointed, we may write $g_j = T_i + h_j$ with $h_j$ 
not depending on $T_i$ and, eliminating $T_i$,
we realize $X$ in a smaller fake weighted projective 
space.
Moreover, suitably renumbering, we achieve
$$
x_0 \le \ldots \le x_n,
\qquad\qquad
u_1 \le \ldots \le u_s.
$$
According to the generality condition, we may assume
that every monomial of degree $\mu_j$ shows up
in the relation $g_j$, where $j = 1, \ldots, s$.
In particular, as Lemma~\ref{lem:fwpsfromlaurent}
shows $\mu_j = l_{ji} w_i$ with $l_{ji} \in \ZZ_{\ge 1}$,
we see that each power $T_i^{l_{ji}}$ is a monomial
of~$g_j$. By irredundance of the presentation,
we have $l_{ji} \ge 2$ for all $i$ and $j$.

We will now establish effective 
bounds on the $w_i$ and $\mu_j$
that finally allow a computational 
treatment of the remaining cases.
The following first constraints 
are caused by terminality.
By Corollary~\ref{cor:SigmaXcrit},
all two-dimensional cones of~$\Sigma$
belong to~$\Sigma_X$ and by
Corollary~\ref{cor:acc-singu}, 
the toric orbits corresponding
to these cones host at most
terminal singularities of $Z$.
Thus, Lemma~\ref{lem:terminalprops}~(i)
tells us that $\Cl(Z)$ is
generated by any $n-1$ of 
$w_0, \ldots, w_n$.
In particular, any $n-1$ of 
$x_0, \ldots, x_n$ are coprime
and, choosing suitable generators
for $\Cl(Z)$, we achieve
$$ 
\Cl(Z) 
\ = \ 
\ZZ \times \ZZ / t_1 \ZZ \times \ldots \times \ZZ / t_q \ZZ,
\qquad
q \ \le \ n-1.
$$

Next, we see how the Fano property of $X$ 
contributes to bounding conditions.
Generality and Corollary~\ref{cor:general2SrRa}
ensure that $X$ inherits its divisor class 
group from the
ambient fake weighted projective space~$Z$.
Moreover, by Proposition~\ref{prop:adjunction},
the anticanonical class~$-\mathcal{K}_X$ of~$X$ 
is given in terms of the generator 
degrees~$w_i = \deg(T_i)$, the relation 
degrees $\mu_j = \deg(g_j)$ and $n = s+3$ as 
$$
-\mathcal{K}_X 
\ = \ 
w_0 + \ldots + w_{n} 
- 
\mu_1 - \ldots - \mu_s 
\ \in \ 
\Cl(Z)
\ = \ 
\Cl(X).
$$
Now, consider the tuples $\xi = (x_0, \ldots, x_n)$
and $(u_1,\ldots,u_s)$ of $\ZZ$-parts of the
generator and relation degrees.
As seen above, we have $u_j= l_{ji}x_i$ with
$l_{ji} \in \ZZ_{\ge 2}$ for all $i$ and $j$.
Thus, $m(\xi)  = \lcm(x_0, \ldots, x_n)$
divides all~$u_j$, in particular $m(\xi) \le u_j$.
Moreover, if $m(\xi) \ne x_n$, then we even
have $2m(\xi) \le u_j$.
Altogether, with $M(\xi) := 2m(\xi)$ if $m(\xi) \ne x_n$
and $M(\xi) := m(\xi)$ else, we arrive
in particular at the
inequality~(\ref{eq:abstractFano}):
$$
(n-3) M(\xi)
\ = \
s M(\xi)
\ \le \ 
u_1 + \ldots + u_s
\ < \
x_0 + \ldots + x_n.
$$

This allows us to conclude that the number $s$
of defining equations for our $X \subseteq Z$
is at most three.
Indeed, inserting $2x_n \le u_j$ and $x_i \le x_n$,
we see that $2sx_n$ is strictly less than
$(n+1)x_n = (s+4)x_n$.
We go through the cases $s = 1, 2, 3$
and provide upper bounds on the generator
degrees $x_0, \ldots, x_{n}$.

Let $s = 1$. Then $n = 4$.
We will show $x_4 \le 41$.
As noted above any three of $x_0, \ldots, x_4$
are coprime.
Thus, Lemma~\ref{lem:wf5-tuples} applies,
showing that we have $x_4 \le 41$ or 
the tuple $(x_0, \ldots, x_4)$ is one of 
$$
(1, 1, 1, x_4, x_4), \quad
(1, 1, 2, x_4, x_4), \quad
(1, 2, 2, x_4, x_4).
$$
In the latter case, consider
$\sigma = \cone(v_0, v_1, v_2) \in \Sigma$.
Corollary~\ref{cor:SigmaXcrit} ensures
$\sigma \in \Sigma_X$.
Due to by Corollary~\ref{cor:acc-singu},
we may apply Lemma~\ref{lem:terminalprops}~(ii),
telling us
$$
x_4
\ = \
\gcd(x_3, x_4)
\ < \
x_0 + x_1 + x_2
\ \le \
5.
$$
Let $s= 2$. Then $n=5$. We will show $x_5 \le 21$.
According to Lemma~\ref{lem:wf6tuples-1},
we only have to treat the case $x_2 = \ldots = x_5$.
As noted above, we have
$$
x_5 \ = \ \gcd(x_2, \dotsc, x_5) \ = \ 1.
$$
Let $s = 3$. Then $n = 6$.
Lemma~\ref{lem:wf7tuples} leaves us with $(x_0, \ldots, x_6)$
being one of the tuples $(1, 1, 1, 1, 1, 1)$
and $(2, 2, 3, 3, 3, 3, 3)$.
As before, we can exclude the second
configuration.

Next, we perform a computational step.
Subject to the bounds just found,
we determine all ordered, well formed
tuples $\xi = (x_0, \ldots, x_n)$,
where $n = s+3$ and $s = 1,2,3$,
that admit an ordered tuple
$(u_1,\ldots,u_s)$ such that
$$
u_1+ \ldots + u_s
 > 
x_0 + \ldots + x_n,
\qquad
l_{ji} :=  \frac{u_j}{x_i}  \in  \ZZ_{\ge 2},
\ j = 1, \ldots, s,
\ i = 0, \ldots, n
$$
holds and any $n-1$ of $x_0, \ldots, x_n$
are coprime.
This is an elementary computation
leaving us with about a hundred
tuples $\xi = (x_0, \ldots, x_n)$, each of
which satisfies $x_0 = 1$.

As consequence, we can bound the data of
the divisor class group $\Cl(Z)$.
As noted, we have $q \le n-1$
and Lemma~\ref{lem:torsbound} now provides
upper bounds on the orders~$t_k$ of
the finite cyclic factors.
This allows us to compute a list of
specifying data $(Q, \mu_1, \ldots, \mu_s)$
of candidates for $X \subseteq Z$
by building up degree maps
$$
Q \colon
\ZZ^{n+2}
\ \to \
\Cl(Z) = \ZZ \times \ZZ / t_1 \ZZ \times \dotsb \times \ZZ / t_{q} \ZZ,
\qquad
e_i \ \mapsto \ w_i
$$
and pick out those that satisfy the
constraints established so far.
In a further step, we check the candidates
for terminality using the criterion provided
Corollary~\ref{cor:acc-singu}; computationally,
this amounts to a search of lattice points
in integral polytopes.
The affirmatively tested candidates form the list 
of Theorem~\ref{thm:terminalWHf-new}.
All the computations have been performed with
Magma and will be made available elsewhere.
%at~\cite{MCM}.%

Remark~\ref{rem:verfication-tools} 
shows that each specifying data
$(Q, \mu)$ in the list 
of Theorem~\ref{thm:terminalWHf-new} stems
indeed from a general toric complete intersection
$X$ in the fake weighted projective space~$Z$.
Finally, Corollary~\ref{cor:general2SrRa} ensures
that the Cox ring of all listed $X$ is as claimed.
In particular, none of the $X$ is toric.
Most of the listed families can be distinguished
via the divisor class group $\Cl(X)$, the anticanonical
self intersection $-\mathcal{K}_X^3$ and $h^0(-\KKK_X)$.
For Numbers~12 and 39, observe that their Cox rings
have non-isomorphic configurations of generator degrees,
which also distinguishes the members of these families.
\end{proof}

\end{document}